\documentclass[11pt, reqno, letterpaper]{amsart}

\usepackage{fancybox,color, mathtools, array}

\usepackage{tikz}
 
\usepackage{latexsym}
\usepackage{color} 
\usepackage{amssymb}
\usepackage{mathrsfs}
\usepackage{amsmath}
\usepackage{fancybox,color,graphicx}
\usepackage{enumerate}
\usepackage[active]{srcltx}
\usepackage{hyperref}

\usepackage{hypernat}
\usepackage{setspace}
\usepackage{soul}
 \newcommand{\eps}{{\varepsilon}}
 
 \def\1{\raisebox{2pt}{\rm{$\chi$}}}
 
\def\vint_#1{\mathchoice%
          {\mathop{\kern 0.2em\vrule width 0.6em height 0.69678ex depth -0.58065ex
                  \kern -0.8em \intop}\nolimits_{\kern -0.4em#1}}%
          {\mathop{\kern 0.1em\vrule width 0.5em height 0.69678ex depth -0.60387ex
                  \kern -0.6em \intop}\nolimits_{#1}}%
          {\mathop{\kern 0.1em\vrule width 0.5em height 0.69678ex
              depth -0.60387ex
                  \kern -0.6em \intop}\nolimits_{#1}}%
          {\mathop{\kern 0.1em\vrule width 0.5em height 0.69678ex depth -0.60387ex
                  \kern -0.6em \intop}\nolimits_{#1}}}
\def\vintslides_#1{\mathchoice%
          {\mathop{\kern 0.1em\vrule width 0.5em height 0.697ex depth -0.581ex
                  \kern -0.6em \intop}\nolimits_{\kern -0.4em#1}}%
          {\mathop{\kern 0.1em\vrule width 0.3em height 0.697ex depth -0.604ex
                  \kern -0.4em \intop}\nolimits_{#1}}%
          {\mathop{\kern 0.1em\vrule width 0.3em height 0.697ex depth -0.604ex
                  \kern -0.4em \intop}\nolimits_{#1}}%
          {\mathop{\kern 0.1em\vrule width 0.3em height 0.697ex depth -0.604ex
                  \kern -0.4em \intop}\nolimits_{#1}}}

\newcommand{\aveint}[2]{\mathchoice%
          {\mathop{\kern 0.2em\vrule width 0.6em height 0.69678ex depth -0.58065ex
                  \kern -0.8em \intop}\nolimits_{\kern -0.45em#1}^{#2}}%
          {\mathop{\kern 0.1em\vrule width 0.5em height 0.69678ex depth -0.60387ex
                  \kern -0.6em \intop}\nolimits_{#1}^{#2}}%
          {\mathop{\kern 0.1em\vrule width 0.5em height 0.69678ex depth -0.60387ex
                  \kern -0.6em \intop}\nolimits_{#1}^{#2}}%
          {\mathop{\kern 0.1em\vrule width 0.5em height 0.69678ex depth -0.60387ex
                  \kern -0.6em \intop}\nolimits_{#1}^{#2}}}

\newtheorem{theorem}{Theorem}[section]
\newtheorem{corollary}[theorem]{Corollary}
\newtheorem{lemma}[theorem]{Lemma}

\newtheorem{definition}[theorem]{Definition}
\newtheorem{remark}[theorem]{Remark}
\newtheorem{example}[theorem]{Example}

 \newcommand{\trace}{\operatorname{trace}}
  \newcommand{\R}{\mathbb{R}}

\newcommand{\xintloo}[1]{\int\limits_{#1} \kern-18pt\raise4pt\hbox to7pt {\hrulefill}\ }   \numberwithin{equation}{section}

\def\Xint#1{\mathchoice
   {\XXint\displaystyle\textstyle{#1}}%
   {\XXint\textstyle\scriptstyle{#1}}%
   {\XXint\scriptstyle\scriptscriptstyle{#1}}%
   {\XXint\scriptscriptstyle\scriptscriptstyle{#1}}%
   \!\int}
\def\XXint#1#2#3{{\setbox0=\hbox{$#1{#2#3}{\int}$}
     \vcenter{\hbox{$#2#3$}}\kern-.5\wd0}}

\def\dashint{\Xint-}

 \DeclareMathOperator{\tr}{trace}

\numberwithin{theorem}{section}
\numberwithin{equation}{section}
\def \A{\mathcal A}
\def \eps{\varepsilon}
\def \R{\mathbb R}

\begin{document}

\title[Asymptotic Mean-Value Formulas for Solutions of General Second-Order PDE]{Asymptotic Mean-Value Formulas for Solutions of General Second-Order Elliptic Equations
}

\thanks{P.B. partially supported by the Academy of Finland project no. 298641. \\
\indent F.C. partially supported by grant MTM2017-84214-C2-1-P and 
PID2019-110712GB-I100 funded by MCIN/AEI/10.13039/501100011033 and by ``ERDF A way of making Europe''.
 \\
\indent J.D.R. partially supported by 
CONICET grant PIP GI No 11220150100036CO
(Argentina), PICT-2018-03183 (Argentina) and UBACyT grant 20020160100155BA (Argentina).
}

\author[P. Blanc]{Pablo Blanc}
\address{Department of Mathematics and Statistics, University of Jyv\"askyl\"a, PO Box 35, FI-40014, \newline
\noindent Jyv\"askyl\"a, Finland}
\email{pblanc@dm.uba.ar}

\author[F. Charro]{Fernando Charro}
\address{Department of Mathematics, Wayne State University, 656 W. Kirby, Detroit, MI 48202, USA}
\email{fcharro@wayne.edu}
\thanks{}

\author[J. J. Manfredi]{Juan J. Manfredi}
\address{
Department of Mathematics,
University of Pittsburgh, Pittsburgh, PA 15260, USA.
}
\email{manfredi@pitt.edu}

\author[J. D. Rossi]{Julio D. Rossi}
\address{Departamento  de Matem\'atica, FCEyN, Universidad de Buenos Aires,
 Pabell\'on I, Ciudad Universitaria (1428),
Buenos Aires, Argentina.}
\email{jrossi@dm.uba.ar}

\newcommand{\pablo}[1]{{\color{blue}{#1}}}
\long\def\comment#1{\marginpar{\raggedright\small$\bullet$\ #1}}

\keywords{Mean-value formulas, viscosity solutions, $k$-Hessian equations, Issacs equations.
\\
\indent 2020 {\it Mathematics Subject Classification:}
35J60, % Nonlinear elliptic equations
% 35J96, % Elliptic Monge-Amp?re equations
35D40, %  Viscosity solutions to PDEs
35B05%	Oscillation, zeros of solutions, mean-value theorems
}
\date{}

\begin{abstract}
We obtain asymptotic mean-value formulas for solutions of second-order elliptic equations.
Our approach is very flexible and allows us to consider several families of operators  obtained as an infimum, a supremum, or a combination of both infimum and supremum, of linear operators. 
The families of equations that we consider include well-known operators such as Pucci, Issacs, and $k$-Hessian operators. 
\end{abstract}

\maketitle

%{\singlespacing
%\tableofcontents        
%}

%\newpage

\section{Introduction}\label{sec.intro}
It is well-known that  mean-value formulas characterize harmonic functions; %For example, we have
%$$
%\Delta u (x) =0 \ \ \mbox{in } \Omega \quad\iff\quad u(x) = \dashint_{B_\varepsilon(x)} u(y)\,dy  \  \ \mbox{for every } B_\varepsilon (x) \subset \Omega.
%$$ 
%Here and in what follows, $\Omega \subset \mathbb{R}^n$ is an open set. 
in fact, a weaker statement, 
known as  the asymptotic mean-value property,  is enough to characterize harmonicity
 (see \cite{Blaschke,Ku,Privaloff}).
 Furthermore,
we have
\[
\Delta u(x)=\lim_{\varepsilon \to0}\frac{2(n+2)}{\varepsilon^2}\left(\dashint_{B_\varepsilon(x)} u(y)\,dy-u(x)\right),
\]
from which we conclude that, when $f$ is continuous, a function $u$ satisfies the classical Poisson equation $\Delta u(x)=f(x)$ in $\Omega$ if and only if 
\begin{equation}
\label{Laplaciano.formula.asintotica}
u(x)=\dashint_{B_\varepsilon (x)} u(y)\,dy-\frac{\varepsilon^2}{2(n+2)}\,f(x)+o(\varepsilon^2)\qquad\textrm{as} \ \varepsilon \to0
\end{equation}
 for each $x \in \Omega$.
We use the standard notation  $o(\varepsilon^2)$ to denote a quantity 
such that $o(\varepsilon^2)/\varepsilon^2 \to 0 $
as $\varepsilon \to 0$.
Analogous results hold for sub- and supersolutions replacing equalities by appropriate inequalities.

The classic mean-value property can be seen as a nonlocal formulation of the local Laplace and Poisson equations.
It has deep connections with other fundamental properties  such as the maximum principle, which allows proving existence results by Perron's method or symmetry of solutions by the moving-planes method. In this sense,  the mean-value property and the maximum principle are nothing more than a quantitative expression of the monotonicity inherent to the equation.

In our recent paper \cite{[Blanc et al. 2020]} we established asymptotic mean-value properties in the viscosity sense for solutions to  the Monge-Amp\`ere equation,
\begin{equation}\label{equation.intro.model}
 \det D^2u (x)=f (x),\quad \textrm{in}\ \Omega,
\end{equation}
which is elliptic only in the class of convex functions and, consequently,  requires $f\geq0$.
Our results in \cite{[Blanc et al. 2020]} are based on the formula
\[
n\big(\det D^2u(x)\big)^{1/n}=\inf_{\det A=1}{\rm trace}(A^tD^2u(x)A),
\]
which holds whenever $D^2u(x)\geq 0$.
Mean-value properties in the viscosity sense require more geometrical arguments adapted to  the operator. In particular, our methods do not require explicit representation formulas, making them flexible enough to be applied to various nonlinear problems.

%Mean-value formulas are a robust and flexible tool in analysis.
In this article, we establish asymptotic mean-value formulas for a wide array of fully nonlinear equations that includes degenerate operators such as the $k$-Hessians, \cite{Ca-Ni-Sp2}, truncated Laplacians, \cite{[Birindeli et al. 2020]}, and prescribed eigenvalues of the Hessian, \cite{BR,OS}. %, as well as general uniformly elliptic operators,  degenerate Isaacs equations, and other
The operators we discuss  can be written as an infimum of linear operators with 
coefficients chosen from a given set, see below for examples. Of course, there are corresponding statements for equations involving a supremum or combinations of infimum and supremum.

Mean-value formulas hold under more lenient regularity conditions than the corresponding PDEs and can provide a simple and unified approach to nonlinear equations in non-Euclidean contexts such as Carnot groups. 
As a proof of concept, we prove mean-value formulas for Monge-Amp\`ere operators in the Heisenberg group in Section \ref{section.heisenberg} below.
Other directions for applications of the asymptotic mean-value formulas below concern  game-theoretic interpretations of the corresponding PDEs, and their numerical analysis.
In this regard, there are convergent difference schemes for the normalized infinity Laplacian and $p$-Laplacian using mean-value formulas, see \cite{[Oberman 2013]}. Hence, the mean-value formulas that we develop here could provide new numerical methods for the corresponding nonlinear equations.

%\subsection{Operators involving unbounded sets of coefficients} 
%\label{s-bounded.99} 
Let us now describe our results. First, we consider differential operators of the form $F(D^2 u)$, with $F:S^n(\R)\to \R\cup\{-\infty\}$ given by
\begin{equation}\label{eq.thm.main3.intro}
F(M)=\inf_{A\in\A} \tr(A^tMA).
\end{equation}
Here, $\A$ is a  subset of $S_+^n(\mathbb{R})$, the set of symmetric positive semi-definite matrices. 
We observe in Lemma ~\ref{Axbounded} below that the set $\A$ is bounded if and only if  the operator $F(M)$ is  well-defined (finite) for all $M\in S^n(\R)$.  
Therefore,  $\A$ determines the set of admissible solutions, functions for which  $F(D^2u)$ is  finite. 
We say that 
\begin{equation} 
u \in C^2(\Omega)\text{ is }  \mathcal{A}\text{-admissible} \text{ in } \Omega \text{ if } D^2 u(x)\in \Gamma_\A \text{ for every } x\in\Omega,
\end{equation}
where
we define the cone
\[
\Gamma_\A=\Big\{M\in S^{n}(\mathbb{R}): F(M)>-\infty \Big\}
\]
(see \cite{Ca-Ni-Sp2} for a related notion of admissibility). 
This condition plays an analogous role to the convexity for the Monge-Amp\`ere equation;
in fact, 
\begin{equation}\label{mongeasinf}
\inf_{A\in\mathcal{A}} \textrm{trace}(A^tD^2u(x)A)
=
\left\{
\begin{aligned}
&n\left(
\det{D^2u(x)}
\right)^{1/n}
 & &\textrm{if}\ D^2u(x)\geq0\\
& -\infty& &\textrm{otherwise},
\end{aligned}
\right.
\end{equation}
for the unbounded set $\mathcal{A}=\{A\in S_+^n(\mathbb{R}): \det(A)=~1\}$, 
 see \cite{[Blanc et al. 2020]}.

Observe that $F$ is an infimum of linear functions, which are continuous; therefore, it is upper semi-continuous.
Even more, $F$ is concave and continuous in $\Gamma_\mathcal{A}^\circ$. We  assume that 
\begin{equation}\label{hyp.intro.semicont}
F\ \textnormal{is  lower semi-continuous in}\ \Gamma_\mathcal{A}\setminus \Gamma_\mathcal{A}^\circ.
\end{equation}
so that we have that
\begin{equation}\label{hyp.intro.cont}
F\ \textnormal{is continuous in}\ \Gamma_\A.
\end{equation}
Condition \eqref{hyp.intro.semicont}   is not satisfied in Example  ~\ref{remark.ejemplo2}  below and the mean-value formula fails.

To deal with operators of the form \eqref{eq.thm.main3.intro}, where the class of matrices $\mathcal{A}$ is unbounded, we  consider matrices $A$ such that $A\in\A$ and
$A\leq \phi(\varepsilon)I$ with $\phi(\eps)$ a positive function such that 
\begin{equation}
\label{hypothesis.phi.intro}
 \lim_{\varepsilon\to0} \phi(\varepsilon)=+\infty
 \qquad
  \textrm{and}
 \qquad
 \lim_{\varepsilon\to0} \varepsilon\,\phi(\varepsilon) =0
\end{equation}
(an example of such a function is  $\phi(\varepsilon)=\varepsilon^{-\alpha}$ for $0<\alpha<1$).
Notice that the condition $A\leq \phi(\varepsilon)I$ becomes less restrictive as $\varepsilon\to0$ but  is still enough to   make  the mean-value formula local, see Section \ref{sec.unboundd}. We obtain the following result.

\begin{theorem}
\label{thm.main3.intro}
Let $\phi(\varepsilon)$ 
be a positive function  that satisfies \eqref{hypothesis.phi.intro}, and
$F:S^n(\R)\to \R\cup\{-\infty\}$ an operator of the form \eqref{eq.thm.main3.intro} that satisfies \eqref{hyp.intro.semicont}.
Let the function
$u\in C^2(\Omega)$ be $\mathcal{A}$-admissible.
Then, for every $x\in\Omega$ we have 
\begin{equation}
\label{eq.thm.main3.intro.MVP}
F(D^2u (x))
=
2(n+2)
\mathop{\mathop{\inf}_{A\in\A}}_
{A\leq \phi(\varepsilon)I}\,
\dashint_{B_{\varepsilon}(0)}
\frac{u(x+Ay)-u(x)}{\varepsilon^2}
\,dy
+
o(1),
\qquad\textrm{as $\eps\to0$.}
\end{equation}
\end{theorem}

As a consequence, we can characterize viscosity solutions to the equation $F(D^2u (x)) =f(x)$ in terms of an asymptotic mean-value formula in the viscosity sense. 
For the precise definition of the notion of viscosity solutions and viscosity mean-value formulas, see Section \ref{sect-prelim} below, \cite{CIL}, and \cite{[Manfredi et al. 2010]}. 
Informally, an equation or a mean-value property
hold in the viscosity sense  when they hold with an appropriate inequality (instead of an equality) for smooth test functions that touch $u$ from above or below at $x$. 

The concept of a mean-value formula in the viscosity sense is weaker than a mean-value formula that holds pointwise.
For the infinity Laplacian and the Monge-Amp\`ere equation, there are instances of asymptotic mean-value formulas that hold in the viscosity sense but do not hold pointwise, see \cite{{[Manfredi et al. 2010]}} and \cite{[Blanc et al. 2020]}, respectively.
An interesting question is, under what circumstances do the viscosity mean-value formulas outlined above hold pointwise. 
This question has an obvious affirmative answer for all equations for which viscosity solutions are known to be classical since, in that case, we can apply the pointwise formula for regular functions  in \eqref{eq.thm.main3.intro.MVP}. %Some examples are convex/concave uniformly elliptic  fully nonlinear equations (as a consequence of Evans-Krylov estimates) but also some Isaacs equations, which are neither convex nor concave, \cite{CC.2003}, which we will cover next. 
Nevertheless, there are examples of non-classical viscosity solutions for which  mean-value formulas hold pointwise; see \cite{Angel.Arroyo.Tesis,[Blanc et al. 2020],LindqvistManfredi}. 
\begin{theorem} \label{th.sol.viscosa.intro.1} 
Let $\phi(\varepsilon)$ 
be a positive function  that satisfies \eqref{hypothesis.phi.intro}.
Consider $f\in C(\Omega)$ and
  $F:S^n(\R)\to \R\cup\{-\infty\}$ defined as in \eqref{eq.thm.main3.intro} that   satisfies \eqref{hyp.intro.semicont}.
Then, a function $u\in C(\Omega)$ is a viscosity subsolution (respectively, supersolution) of 
\[
F(D^2u(x)) =f(x) \quad \mbox{in } \Omega,
\]
if and only if
\[
u(x)
\leq
\mathop{\mathop{\inf}_{A\in\A}}_
{A\leq \phi(\varepsilon)I}
\dashint_{B_{\varepsilon}(0)}
u(x+Ay)
\,dy
-
\frac{\varepsilon^2}{2(n+2)}
f(x)
+
o(\varepsilon^2),
\qquad\textrm{as $\eps\to0$}
\]
(respectively, $\geq$)
 in the viscosity sense.
\end{theorem}

\begin{remark} {\rm 
 Analogous results hold for supremum operators 
$$\displaystyle
F(D^2u (x))=\sup_{A\in\A} \tr(A^tD^2u(x)A)$$
 and operators with a combination of an infimum and a supremum. }
\end{remark}

A fundamental example of application of Theorems~\ref{thm.main3.intro} and \ref{th.sol.viscosa.intro.1}  are the $k$-Hessian operators,  given by the elementary  symmetric polynomials
\[
\sigma_k(\lambda_1,\dots,\lambda_n)=\sum_{1\leq i_1< i_2<\cdots<i_k\leq n}\lambda_{i_1}\lambda_{i_2}\dots\lambda_{i_k}
\]
evaluated in the eigenvalues of the Hessian, $\big\{\lambda_i (D^2u)\big\}_{1\leq i \leq n}$. Here,   $k=1,2,\ldots,n$ and the cases $k=1$ and $k=n$ correspond to the Laplacian and  Monge-Amp\`ere, respectively.
For these operators to fit our framework we need to write them in the  form 
\[%\begin{equation}\label{k.hessian.rewritten0}
 F_k(D^2u (x))= k\, \sigma_k\Big(\lambda_1(D^2u(x)),\ldots,\lambda_n(D^2u(x))\Big)^{\frac{1}{k}},
\]%\end{equation}
according to the following characterization, which we prove in Section \ref{section.k.hess}.
\begin{lemma}\label{def.k2}
Let $k\geq2$ and
\[
\Gamma_k=\Big\{\lambda\in\R^n: \sigma_j(\lambda)> 0\ \text{for all}\ j=1,\dots,k \Big\}.
\]
Then, for every $M\in S^n(\R)$ we have
\begin{equation}
\label{eq.main.k.99}
\inf_{A\in\mathcal{A}_k}\tr(A^tMA)
=
\left\{
\begin{aligned}
&k\, \sigma_k(\lambda(M))^{\frac{1}{k}}
 & &\textrm{if}\ M\in \overline \Gamma_k,\\
& -\infty& &\textrm{otherwise},
\end{aligned}
\right.
\end{equation}
where
\[
\A_k=
\Big\{
A\in S_+^n(\mathbb{R}) : \ \lambda_i^2(A)=\sigma_{k-1,i}(\gamma) \text{ with } \gamma=(\gamma_1,\dots,\gamma_n)\in  \Gamma_k \text{ and } \sigma_k(\gamma)=1
\Big\},
\]
and
$\sigma_{k-1,i}(\gamma_1,\dots,\gamma_n)=\sigma_{k-1}(\gamma_1,\dots,\gamma_{i-1},0,\gamma_{i+1},\dots,\gamma_n).$
\end{lemma}

An important consequence of Lemma \ref{def.k2} is that
for the $k$-Hessian operators,  $\mathcal{A}$-admissibility is equivalent to the notion of $k$-convexity in \cite{Trudinger.Wang.1997,Trudinger.Wang.1999,Trudinger.Wang.2002}, see  Corollary \ref{coro.GammaA=Gammak}. This means that
\[
\Gamma_{\A_k}=\big\{M\in S^{n}(\mathbb{R}): F(M)>-\infty \big\}=\big\{M\in S^{n}(\mathbb{R}):F(M)\geq 0\big\}
=\overline{\Gamma}_k
\]
and
\begin{equation}
\label{hypo.F=0.intro}
F\equiv 0 \quad \mbox{on }\partial\Gamma_{\A_k}.
\end{equation}
A condition such as \eqref{hypo.F=0.intro} could be used in place of \eqref{hyp.intro.semicont} to prove the mean-value property. In fact,
 \eqref{hypo.F=0.intro} may  even appear more natural at first; for instance, it was used in \cite{HL1} 
in relation to the existence and uniqueness of solutions 
to general fully nonlinear second-order PDEs. 
However, condition  \eqref{hypo.F=0.intro} implies \eqref{hyp.intro.semicont}, but not vice versa (see Remark \ref{remark.F.geq.C} and Example \ref{remark.ejemplo1} below), making condition 
\eqref{hyp.intro.semicont} more general.

In this case Theorem \ref{thm.main3.intro} reads as follows, see Section \ref{section.k.hess} for details.

\begin{theorem}
\label{thm.k-hessians}
Let  $\phi(\varepsilon)$ 
be a positive function  that satisfies \eqref{hypothesis.phi.intro}  and   
assume that $u\in C^2(\Omega)$ is $k$-convex, that is, $\sigma_j(\lambda(D^2 u(x))) \geq  0$ for all $ j=1,\dots,k$, 
for every $x\in\Omega$. 
Then, for every $x\in\Omega$ we have 
\begin{equation}
\label{eq.thm.main3.intro.MVP.999}
u(x)
=
\mathop{\mathop{\inf}_{A\in\A_k}}_
{A\leq \phi(\varepsilon)I}
\dashint_{B_{\varepsilon}(0)}
u(x+Ay)
\,dy
-
\frac{\varepsilon^2}{2(n+2)}
k (\sigma_k(D^2u (x)))^{\frac1k}
+
o(\varepsilon^2), \qquad\textnormal{as $\eps\to0$,}
\end{equation}
with a corresponding viscosity characterization as in Theorem \ref{th.sol.viscosa.intro.1}.
\end{theorem}

%  It is straightforward to show a corresponding result characterizing viscosity solutions to the equation $F_k(D^2u(x)) =f(x)$ in $\Omega$ 
%by an asymptotic mean-value formula
%in the viscosity sense. Because it is completely analogous to Theorem \ref{th.sol.viscosa.intro} above, we will not state it explicitly.

 It is also possible to consider operators where the class of matrices $\A$ depends on the point $x$. This happens naturally when we consider mean-value formulas for Monge-Amp\`ere operators in the Heisenberg group in Section \ref{section.heisenberg}.
In that case, the sets $\A$ depend on the point $x$ and are unbounded.

We also obtain a  rich family of examples when considering operators for which $\A_x\subset S_+^n(\mathbb{R})$ is bounded for each $x\in\Omega$. In fact, we consider  sup-inf operators where for every $x$, the supremum is taken over a subset $\mathbb{A}_x$ of $\mathcal P (S_+^n(\mathbb{R}))$, the power set of $S_+^n(\mathbb{R})$. This is not done only for the sake of generality. It is motivated by examples that cannot be covered otherwise, such as prescribed eigenvalues of the Hessian and Isaacs operators. These operators can be degenerate elliptic since we do not impose any lower bounds on the eigenvalues of the matrices
$A\in\mathcal{A}_x$. We provide a list of examples  below which shows the flexibility of the approach.

Therefore, let us now consider differential operators of the form $F\big(x,D^2u(x)\big)$ with $F:\Omega \times S^n(\R)\to \R$ given by
\begin{equation}\label{isasc.pablo.intro}
F(x,M)=\sup_{\mathcal{A}\in\mathbb{A}_x} \inf_{A\in \mathcal{A}} \textrm{trace}(A^tMA).
\end{equation}
Here,
 $\mathbb{A}_x\subset \mathcal P (S_+^n(\mathbb{R}))$ (the power set of $S_+^n(\mathbb{R})$) is a non-empty subset for each $x\in\mathbb{R}^n$ and we assume that
%the union of all these sets,
\begin{equation}\label{union-bounded}
\bigcup\mathbb{A}_x
=
\Big\{
A\in S_+^n(\mathbb{R})\ : \
A\in \mathcal{A}\ \textrm{for some}\ \mathcal{A}\in\mathbb{A}_x
\Big\}
\quad\textrm{is bounded.}
\end{equation}
Observe that if $\mathbb{A}_x$  contains only one element 
 $\A_x$ for each $x$, then 
\eqref{isasc.pablo.intro} is equivalent to
\begin{equation}\label{operator.bounded.A}
F(x,M)=\inf_{A\in\A_x} \tr(A^tMA),
\end{equation}
with  $\A_x\subset S_+^n(\mathbb{R})$  bounded  for each $x\in\mathbb{R}^n$.
On the other hand, if  every $\mathbb{A}_x$ is a set of singletons, \eqref{isasc.pablo.intro} is equivalent to
\begin{equation}\label{operator.bounded.A.sup}
F(x,M)=\sup_{A\in\A_x} \tr(A^tMA),
\end{equation}
where $\A_x=
\{
A\in S_+^n(\mathbb{R}) : 
A\in \mathcal{A}\ \textrm{for some}\ \mathcal{A}\in\mathbb{A}_x
\}$. One can also consider  inf-sup operators of the form 
\begin{equation}
	\label{isasc.pablo.intro.interchanged}
F(x,M)=\inf_{\mathcal{A}\in\mathbb{A}_x} \sup_{A\in \mathcal{A}} \textnormal{trace}(A^tMA)
\end{equation}
with straightforward adaptations in the statements and proofs.

 By Lemma \ref{Axbounded},  \eqref{union-bounded} is  equivalent to  the operator $M\mapsto F(x,M)$ being  well-defined and finite for every  $M\in S^n(\R)$; in particular, % $\Gamma_{\A_x}\equiv S^{n}(\mathbb{R})$ and 
 every $u \in C^2(\Omega)$ is admissible. Moreover, we prove in Lemma \ref{Lip-cond} that % if $\A_x\subset S_+^n(\mathbb{R})$  is bounded  for each $x\in\Omega$, then 
$M\mapsto F(x,M)$ is Lipschitz continuous for every $x\in\Omega$.

We obtain the following counterparts of Theorems \ref{thm.main3.intro} and \ref{th.sol.viscosa.intro.1}. Observe that 
when $\cup\mathbb{A}_x$ is  bounded, the condition    $A\leq \phi(\varepsilon)I$ in \eqref{eq.thm.main3.intro.MVP} becomes unnecessary because  it is satisfied at every point for $\varepsilon$ small enough.

\begin{theorem}
\label{thm.main1}
Consider $u\in C^2(\Omega)$ and  let
$F:\Omega\times S^n(\R)\to \R$ be an operator of the form \eqref{isasc.pablo.intro} that satisfies \eqref{union-bounded}. Then, for every $x\in\Omega$ we have 
\begin{equation}\label{intro.MVP.sup.inf.C2}
F\big(x,D^2u (x)\big)
=
2(n+2)
\sup_{\mathcal{A}\in\mathbb{A}_x} \inf_{A\in \mathcal{A}}\;
\dashint_{B_{\varepsilon}(0)}
\frac{u(x+Ay)-u(x)}{\varepsilon^2}
\;dy
+
o(1),
\quad
\textrm{as $\eps\to0$.}
\end{equation}
\end{theorem}

We have a corresponding result in the line of Theorem \ref{th.sol.viscosa.intro.1}, characterizing viscosity solutions to the equation $F\big(x,D^2u (x)\big) =f(x)$ 
by an asymptotic mean-value formula
in the viscosity sense.

\begin{theorem} \label{th.sol.viscosa.intro.isaacs} 
Consider $f\in C(\Omega)$ and
  $F: S^n(\R)\to \R$ defined as in \eqref{isasc.pablo.intro} that   satisfies \eqref{union-bounded}.
Then, a function $u\in C(\Omega)$ is a viscosity subsolution (respectively, supersolution) of 
\[
F(D^2u(x)) =f(x) \quad \mbox{in } \Omega,
\]
if and only if
\[
u(x)
\leq
\sup_{\mathcal{A}\in\mathbb{A}_x} \inf_{A\in \mathcal{A}}\;
\dashint_{B_{\varepsilon}(0)}
u(x+Ay)
\,dy
-
\frac{\varepsilon^2}{2(n+2)}
f(x)
+
o(\varepsilon^2),
\quad\textrm{as $\eps\to0$}
\]
(respectively, $\geq$)
 in the viscosity sense.
\end{theorem}

Let us now provide some examples of operators of the form  \eqref{isasc.pablo.intro}, \eqref{operator.bounded.A}, and \eqref{operator.bounded.A.sup} 
and the corresponding sets $\mathbb{A}_x$ and $\A_x$. In the sequel, we consider the eigenvalues of a matrix $M \in S^n(\mathbb{R})$ arranged in increasing order, that is, 
 $\lambda_1(M)\leq \lambda_2(M)\leq \cdots \leq \lambda_n(M)$. 

\begin{enumerate}%\itemsep2pt

\item Operators of the form \eqref{isasc.pablo.intro}, \eqref{isasc.pablo.intro.interchanged} include the usual Isaacs operators (see \cite{CC}) defined as
\[%\begin{equation}\label{normal.Isaacs}
F(x,M)=\sup_{\alpha\in \mathcal A} \inf_{\beta\in \mathcal B} \textnormal{trace}\left(A_{\alpha\beta}^tMA_{\alpha\beta}\right)
\qquad\textnormal{and}\qquad
F(x,M)=\inf_{\alpha\in \mathcal A} \sup_{\beta\in \mathcal B} \textnormal{trace}\left(A_{\alpha\beta}^tMA_{\alpha\beta}\right)
\]%end{equation}
for a given family of matrices $\{A_{\alpha\beta}\}_{\alpha\in\mathcal A, \beta\in \mathcal B}$ (and similarly in the inf-sup case). Here we take $\mathbb{A}=\big\{\{A_{\alpha\beta}\}_{\beta\in \mathcal B}\big\}_{\alpha\in\mathcal A}$ in \eqref{isasc.pablo.intro}, with $\mathbb{A}$ independent of the point   $x$. In fact, operators of the form \eqref{isasc.pablo.intro} could be seen as Issacs operators, but %there is no natural choice for $\mathcal B$ and therefore
 the usual Isaacs condition 
\[
\inf_{\alpha\in\mathcal A} \sup_{\beta\in\mathcal B}  \textnormal{trace}\left(A_{\alpha\beta}^tMA_{\alpha\beta}\right)=\sup_{\beta\in\mathcal B} \inf_{\alpha\in\mathcal A}  \textnormal{trace}\left(A_{\alpha\beta}^tMA_{\alpha\beta}\right)
\]
does not seem to have a clear counterpart.
A typical requirement in the literature is that  the family $\{A_{\alpha\beta}\}_{\alpha\in\mathcal A, \beta\in \mathcal B}$ is  uniformly elliptic, i.e., it is uniformly bounded between two constants $0<\theta<\Theta$. In this case, the resulting  Issacs operator is uniformly elliptic. We want to emphasize that here we are only requiring $\cup\mathbb{A}_x$ to be bounded and, therefore,  the operators  \eqref{isasc.pablo.intro}  may be degenerate elliptic.\medskip

\item We  consider  the truncated Laplacians \cite{[Birindeli et al. 2020]}, 
 defined as
\[%\begin{equation}\label{def1}
\mathcal{P}^-_{k}(D^2u)=\sum_{i=1}^k\lambda_i(D^2u) \qquad\text{and}\qquad \mathcal{P}^+_{k}(D^2u)=\sum_{i=1}^k\lambda_{n+1-i}(D^2u),
\]%end{equation}
for  $k=1,2,\ldots,n-1$ (for $k=n$ these operators coincide with the Laplacian). These degenerate operators
 appear naturally in geometric problems, when considering manifolds of partially positive curvature \cite{Sha,Wu}, or mean curvature flow in arbitrary codimension \cite{AS}, see also  \cite{HL1,HL2,CLN1,[Birindeli et al. 2020]}  and the references therein.  The operators $\mathcal{P}^-_{k}$ and  $\mathcal{P}^+_{k}$ are of the form \eqref{operator.bounded.A} and \eqref{operator.bounded.A.sup}, respectively,  for the set
\[
\mathcal{A}=\big\{A\in S_+^n(\mathbb{R}): \lambda_1=\cdots=\lambda_{n-k}=0 \text{ and } \lambda_{n-k+1}=\cdots=\lambda_n=1\big\}.
\]

\item Operators of the form \eqref{isasc.pablo.intro} allow us to consider degenerate operators such as the $k$-th smallest eigenvalue of the Hessian, given by the  Courant--Fischer min-max principle
\begin{equation} \label{Courant-Fisher}
\lambda_{k}\big(D^2u(x)\big)= \max_V \left\{ \min_{v \in V,\ |v|=1}  \langle D^2u(x)v,v\rangle    :  V\subset\mathbb{R}^n\ \text{subspace},\ \textrm{dim}(V)=n-k+1 \right\}.
\end{equation}
We can write the operator $\lambda_k(D^2u)$ in the form
\eqref{isasc.pablo.intro} for the set
\[
\begin{split}
 \mathbb{A}=\displaystyle \Big\{\big\{A\in S_+^n(\mathbb{R}): \lambda_i(A)=0 \text{ for } i\neq n, \lambda_n(A)=1, \text{ and } v_n\in V\big\} &
:\\
 V\subset\mathbb{R}^n \text{ is a subspace of dimension }& n-k+1\Big\}, 
\end{split}
\]
where $v_n$ is the eigenvector corresponding to $\lambda_n(A)$. 
The cases $k=1$ and $k=n$ were 
 studied in \cite{BR, HL1, OS} in connection with the convex and concave envelope of a function; i.e.,  the unique viscosity solutions of $\lambda_1\big(D^2u(x)\big)=0$ and $\lambda_n\big(D^2u(x)\big) = 0$ are, respectively, the convex and concave envelopes of $u|_{\partial\Omega}$ in $\Omega$.
 These operators are of the form \eqref{operator.bounded.A}, \eqref{operator.bounded.A.sup} 
with the set of matrices
 $$\mathcal{A}=\Big\{A\in S_+^n(\mathbb{R}): \lambda_1(A)=\cdots=\lambda_{n-1}(A)=0 \text{ and } \lambda_n(A)=1\Big\}.$$

\item When the mean-value formula involves averages over balls that are not centered at $0$ but at 
$\varepsilon^2 v$ with $|v|=1$, we obtain operators with first-order terms. For example, we have 
\[
\inf_{A\in\A_x}\; 
\dashint_{B_{\varepsilon} (\varepsilon^2 v)}
u(x+Ay)
\,dy
-
u(x)
= \varepsilon^2 
\inf_{A\in\A_x} \left\{\frac{1}{2(n+2)} \tr(A^tD^2u(x)A)
+ \langle Du(x), A v \rangle \right\}
+ o(\varepsilon^2),
\]
as $\eps\to0$.
%\[
%\begin{split}
%\inf_{A\in\A_x}\; &
%\dashint_{B_{\varepsilon} (\varepsilon^2 v)}
%u(x+Ay)
%\,dy
%-
%u(x)
%\\
%&= \varepsilon^2 
%\inf_{A\in\A_x} \left\{\frac{1}{2(n+2)} \tr(A^tD^2u(x)A)
%+ \langle Du(x), A v \rangle \right\}
%+ o(\varepsilon^2),\quad \textnormal{as $\eps\to0$.}
%\end{split}
%\]
We can also look for zero-order terms and consider mean-value properties like
\[
\inf_{A\in\A_x}
(1-\alpha \varepsilon^2)\, \dashint_{B_{\varepsilon} (0)}
u(x+Ay)
\,dy
- u(x)
= \varepsilon^2 \left\{
\frac{1}{2(n+2)}
\inf_{A\in\A_x} \tr(A^tD^2u(x)A)
- \alpha u(x) \right\}
+ o(\varepsilon^2),
\]
as $\eps\to0$.
%\[
%\begin{split}
%\inf_{A\in\A_x}&
%(1-\alpha \varepsilon^2)\, \dashint_{B_{\varepsilon} (0)}
%u(x+Ay)
%\,dy
%- u(x)
%\\
%&= \varepsilon^2 \left\{
%\frac{1}{2(n+2)}
%\inf_{A\in\A_x} \tr(A^tD^2u(x)A)
%- \alpha u(x) \right\}
%+ o(\varepsilon^2),\quad\textnormal{as $\eps\to0$.}
%\end{split}
%\]
\item Finally, we consider extremal  Pucci operators for given ellipticity constants
 $0<\theta<\Theta$, defined as
\[
\mathcal{M}^-_{\theta,\Theta}(D^2u)=\theta\sum_{\lambda_i(D^2u)>0}\lambda_i(D^2u)+\Theta\sum_{\lambda_i(D^2u)<0}\lambda_i(D^2u)
\]
and
\[
\mathcal{M}^+_{\theta,\Theta}(D^2u)=\Theta\sum_{\lambda_i(D^2u)>0}\lambda_i(D^2u)+\theta\sum_{\lambda_i(D^2u)<0}\lambda_i(D^2u).
\]
These operators are of the form \eqref{operator.bounded.A}, \eqref{operator.bounded.A.sup} with
\[
\mathcal{A}_{\theta\Theta}=\left\{A\in S_+^n(\mathbb{R}): \sqrt\theta \leq \lambda_i(A)\leq \sqrt\Theta\right\},
\]
which is bounded uniformly in $x$. %In fact, one can write
%$$
%\mathcal{M}^-_{\theta,\Theta}(M)=\inf_{A\in\A} \tr(A^tMA)
%\qquad
%\textrm{and}
%\qquad
%\mathcal{M}^+_{\theta,\Theta}(M)=\sup_{A\in\A} \tr(A^tMA).
%$$
From the definition of $  \mathcal{A}_{\theta\Theta}$, it is clear that $\mathcal{M}^\pm_{\theta,\Theta}$ are uniformly elliptic operators.
We apply similar ideas to general  fully nonlinear uniformly elliptic operators $F\big(x,D^2u(x)\big)$ (with ellipticity constants $0<\theta\leq\Theta$) by means of the characterization
\[
F(x,M)=\underset{N\in S^n(\R)}{\sup}\inf_{A\in\mathcal{A}_{\theta\Theta}}\Big\{\textnormal{trace}\left(A^tMA\right)+F(x,N)-\textnormal{trace}\left(A^tNA\right)\Big\},
\]
see  Lemma \ref{lemma.every.unif.elliptic.is.Isaacs}.
% From there, we obtain  the following formula
%\[
%\begin{split}
%& F\big(x,D^2u(x)\big)\\
%& =
%\frac{2(n+2)}{\varepsilon^2}
%\underset{N\in S^n(\mathbb{R})}{\inf}
%\sup_{A\in\A}
%\left(
%\dashint_{B_{\varepsilon}(0)}
%\big(
%u(x+Ay)-u(x)
%\big)
%\,dy
%+F(x,N)-\textnormal{trace}\left(A^tNA\right)
%\right)
%+
%o(1),
%\end{split}
%\]
%as $\eps\to0$, which yields a mean-value formula in the viscosity sense. 
%It is interesting that in this formula  $F$ is merely evaluated at constant matrices $N$ in the ``coupling term". Because the formula does not require linearization or differentiation of $F$, it holds for a general uniformly elliptic  equations. 
%%It is worth noting that in general, viscosity solutions to uniformly elliptic equations that are neither concave nor convex are only $C^{1,\alpha}$ and not necessarily $C^2$, see \cite{CC}.
%Observe that for some operators, such as the Pucci and Isaacs operators, we can obtain  different types of mean-value formulas.

\end{enumerate}

The paper is organized as follows: In Section  \ref{sect-prelim} we gather some definitions and preliminary results
and in Section  \ref{sec.infimums} we deal with concave and Isaacs  operators, when the sets of coefficients are bounded. These include general uniformly elliptic operators and operators including lower-order terms. 
 We consider  bounded coefficients first for the sake of clarity, to establish the techniques and ideas before we dive into the unbounded case in Section  \ref{sec.unboundd}. In  Section  \ref{section.k.hess} we study the $k$-Hessians.  Finally, 
Section  \ref{section.heisenberg} is devoted to examples in the Heisenberg group, 
where the class $\mathcal{A}_x$ is unbounded and naturally depends on $x$. 
\

\section{Preliminaries} \label{sect-prelim}

We begin by stating the definition of a viscosity solution to a fully nonlinear, second-order, elliptic PDE.
We refer to
\cite{CIL} for general results on viscosity solutions.
Given a continuous function
\[
\mathcal{F}:\Omega\times\R\times\R^n\times
S^n(\mathbb{R})\to\R,
\]
we  consider the PDE 
\begin{equation}
\label{eqvissol}
\mathcal{F}(x,u (x), Du (x), D^2u (x)) =0, \qquad x \in \Omega.
\end{equation}
Viscosity solutions  use the monotonicity of $\mathcal{F}$ in $D^2u$ (ellipticity) in order to 
``pass derivatives to smooth test functions''. 

\begin{definition}
\label{def.sol.viscosa.1}
A lower semi-continuous function $ u $ is a viscosity
supersolution of \eqref{eqvissol} if for every $ \phi \in C^2$ such that $ \phi $
touches $u$ at $x \in \Omega$ strictly from below (that is, $u-\phi$ has a strict minimum at $x$ with $u(x) = \phi(x)$), we have
$$\mathcal{F} (x,\phi(x),D \phi(x),D^2\phi(x))\leq 0.$$

An upper semi-continuous function $u$ is a subsolution of \eqref{eqvissol} if
for every $ \phi \in C^2$ such that $\phi$ touches $u$ at $ x \in
\Omega$ strictly from above (that is, $u-\phi$ has a strict maximum at $x$ with $u(x) = \phi(x)$), we have
$$\mathcal{F}(x,\phi(x),D \phi(x),D^2\phi(x))\geq 0.$$

Finally, $u$ is a viscosity solution of \eqref{eqvissol} if it is both a super- and a subsolution.
\end{definition}

We will also need the definition of an asymptotic mean-value formula in the viscosity sense.
In the next definition, $M (u, \varepsilon ) (x)$ stands for a mean-value operator (that depends on the parameter
$\varepsilon$) applied to $u$ at the point $x$. For example, we can take 
$$
M (u, \varepsilon ) (x) = \inf_{A\in\A_x}
\dashint_{B_{\varepsilon}(0)}
u(x+Ay)
\,dy
-\frac{\varepsilon^2}{2(n+2)}
f(x)
$$
as in Theorem \ref{th.sol.viscosa.intro.1} and the next section. 

\begin{definition} \label{def.sol.viscosa.asymp}
A continuous function $u$ verifies
$$
u(x) \geq M (u, \varepsilon ) (x) + o(\varepsilon^2), \quad \mbox{as }\varepsilon \to 0,
$$
\emph{in the viscosity sense} if for every $ \phi\in C^{2}$ such that $ u-\phi $ has a strict
minimum at the point $x \in \overline \Omega$  with $u(x)=
\phi(x)$ (i.e., $\phi $
touches $u$ at $x \in \Omega$ strictly from below), we have
$$
\phi (x) \geq   M (\phi, \varepsilon ) (x) + o(\varepsilon^2).
$$
Similarly,
a continuous function  $u$  verifies
$$
u(x) \leq M (u, \varepsilon ) (x) + o(\varepsilon^2), \quad \mbox{as }\varepsilon \to 0,
$$
\emph{in the viscosity sense} if
for every $ \phi \in C^{2}$ such that $ u-\phi $ has a
strict maximum at the point $ x \in \overline{\Omega}$ with $u(x)=
\phi(x)$ (i.e., $\phi $
touches $u$ at $x \in \Omega$ strictly from above), we have
$$
\phi (x) \leq  M (\phi, \varepsilon ) (x)
 + o(\varepsilon^2).
$$
\end{definition}

Next, we include a lemma that is related to the simplest mean-value property 
for the usual Laplacian. A proof can be found in \cite{[Blanc et al. 2020]}.

\begin{lemma}\label{lemma.trace.integral}
	Let $M$ be a square matrix of dimension $n$. Then,
\[
		{\rm trace}(M)
		=
		\frac{n}{\varepsilon^2}
		\dashint_{\partial B_\varepsilon (0)}
		\langle My,y\rangle
		\,d\mathcal{H}^{n-1}(y)=
		\frac{n+2}{\varepsilon^2}
		\dashint_{B_\varepsilon (0)}
		\langle My,y\rangle
		\,dy.
\]
\end{lemma}

\begin{remark} {\rm 
We will use the solid mean 
\[
		{\rm trace}(M)
		=
		\frac{n+2}{\varepsilon^2}
		\dashint_{B_\varepsilon (0)}
		\langle My,y\rangle
		\,dy,
\]
in our proofs of the mean-value formulas. 
However, if one uses the mean on spheres
\[
		{\rm trace}(M)
		=
		\frac{n}{\varepsilon^2}
		\dashint_{\partial B_\varepsilon (0)}
		\langle My,y\rangle
		\,d\mathcal{H}^{n-1}(y)
		\]
		one can obtain mean-value formulas of the type
\begin{equation}
\label{eq.thm.main.bounded.rem}
\inf_{A\in\A_x}
\dashint_{\partial B_{\varepsilon}(0)}
u(x+Ay)
\,dy
-
u(x)
=
\frac{\varepsilon^2}{2 n}
F\big(x,D^2u (x)\big)
+
o(\varepsilon^2),
\qquad\textrm{as $\eps\to0$,}
\end{equation}
for $u\in C^2(\Omega)$ and $F(x,D^2u)$ defined as in \eqref{operator.bounded.A} (and accordingly for the rest of the cases).}
\end{remark}

\section{Bounded Operators}\label{sec.infimums}

In this section we prove Theorems  ~\ref{thm.main1} and \ref{th.sol.viscosa.intro.isaacs}. We consider  differential operators  given by
\begin{equation}\label{isasc.pablo}
	F\big(x,D^2u (x)\big)=\sup_{\mathcal{A}\in\mathbb{A}_x} \inf_{A\in \mathcal{A}} \textrm{trace}(A^tD^2u(x)A),
\end{equation}
where $\mathbb{A}_x\subset \mathcal P (S_+^n(\mathbb{R}))$ (the power set of $S_+^n(\mathbb{R})$) is a non-empty subset for each $x\in\mathbb{R}^n$ and
\begin{equation}\label{assumption.Ax.is.bounded}
\bigcup\mathbb{A}_x
=
\Big\{
A\in S_+^n(\mathbb{R})\ : \
A\in \mathcal{A}\ \textrm{for some}\ \mathcal{A}\in\mathbb{A}_x
\Big\}\quad\textrm{ is bounded.}
\end{equation}
Due to \eqref{assumption.Ax.is.bounded}, the operator  $F(x,M)$ is finite for every $M\in  S^n(\R)$.

%\begin{remark}	\label{remark.local} {\rm
	Moreover, since the set of matrices $\cup\mathbb{A}_x$ is bounded, the mean-value formula \eqref{intro.MVP.sup.inf.C2} is local.
	In fact, for every $x\in\Omega$ and $y\in B_\varepsilon(0)$ there exists $C_x>0$ such that $A\leq C_x I$ for every $A\in\cup\mathbb{A}_x$. We get
	\[
	\textrm{dist}(x+Ay,x)=|Ay|\leq C_x\, \varepsilon \to 0
	\]
	as $\eps\to 0$.
	In particular, observe that for $\eps$ small enough $x+Ay\in\Omega$ for every $y\in B_\varepsilon(0)$
	and hence the integrals 
	$$
	\dashint_{B_{\varepsilon}(0)} u(x+Ay) \,dy
	$$ are well-defined for integrable functions $u:\Omega \to \mathbb{R}$.
%	}
%\end{remark}

\begin{remark}%\label{remark.polar.decomposition} 
{\rm
	We can  assume that the matrices $A\in\cup\mathbb{A}_x$ are symmetric and positive semi-definite without loss of generality.
	This is because, given $A\in M^{n\times n}(\mathbb{R})$, we can write its left polar decomposition $A = SQ$ with $Q$  orthogonal and $S$  positive semi-definite and symmetric. Then, for every $M\in M^{n\times n}(\mathbb{R})$, we have
	\[
	\tr(A^tMA)
	=\tr(Q^tS^tMSQ)=\tr(S^tMSQQ^t)=\tr(S^tMS)
	\]
	 and 
	\[
	\dashint_{B_{\varepsilon}(0)} u(x+Ay) \,dy
	= \dashint_{B_{\varepsilon}(0)} u(x+SQy) \,dy
	= \dashint_{B_{\varepsilon}(0)} u(x+Sz) \,dz.
	\]	
	}
\end{remark}

%This can be compared with the Monge-Amp\`ere case, where $$\mathcal{A}=\{A\in S_+^n(\mathbb{R}): \det(A)=~1\}$$ is unbounded and 
%\[
%\inf_{A\in\mathcal{A}} \textrm{trace}(A^tMA)
%=
%\left\{
%\begin{aligned}
%&n\left(
%\det{M}
%\right)^{1/n}
% & &\textrm{if}\ M\geq0\\
%& -\infty& &\textrm{otherwise},
%\end{aligned}
%\right.
%\]
% see \cite{[Blanc et al. 2020]} and Sections \ref{sec.unboundd} and \ref{section.k.hess} below.

%Therefore, since concave functions can be written as the infimum of affine functions, the class of operators given by \eqref{operator.bounded.A.in.section} coincides with the class of concave operators that are defined for every matrix.

%In the next lemma we handle the continuity of  $F(x,M)$ under perturbations on $M$.

In the next lemma we prove an explicit continuity estimate that we  use in the proof of Theorem \ref{thm.main1}.
\begin{lemma}
\label{Lip-cond}
Consider the differential operator $F\big(x,D^2u(x)\big)$ given by  \eqref{isasc.pablo}.
If $\cup\mathbb{A}_x$ is bounded, then the mapping $M\mapsto F(x,M)$ is Lipschitz continuous in $M$.
In particular,
\begin{equation}
\label{eq.hypo}
\sup_{\mathcal{A}\in\mathbb{A}_x} \inf_{A\in \mathcal{A}} 
\tr(A^t\left(M\pm\eta I\right)A)
\to
\sup_{\mathcal{A}\in\mathbb{A}_x} \inf_{A\in \mathcal{A}} 
\tr (A^tMA)
\end{equation}
as $\eta\to 0$, for every $M\in S^{n}(\R)$.
\end{lemma}

\begin{proof}
We fix $x\in \R^n$.
For every $A\in\cup\mathbb{A}_x$, there exists $C_x>0$ such that $A\leq C_x I$.
Given $M,N\in S^{n}(\R)$, we have 
\[
\begin{split}
&\left|
\sup_{\mathcal{A}\in\mathbb{A}_x} \inf_{A\in \mathcal{A}}  \tr(A^tMA)
-
\sup_{\mathcal{A}\in\mathbb{A}_x} \inf_{A\in \mathcal{A}}  \tr(A^tNA)
\right|
\\
&\quad\leq
\sup_{A\in\cup\mathbb{A}_x} |\tr(A^tMA)-\tr(A^tNA)|
\leq
\sup_{A\in\cup\mathbb{A}_x} |\tr(A^t(M-N)A)|
\leq
n\, C_x^2 \, \|M-N\|,
\end{split}
\]
and the result follows.
\end{proof}

We are now ready to prove Theorem \ref{thm.main1}.
\begin{proof}[Proof of Theorem \ref{thm.main1}]
Given $x\in\Omega$, let us consider the paraboloid 
\begin{equation}\label{definition.P.main}
P(z)=u(x)+\langle\nabla u(x),z-x\rangle+\frac12\langle D^2u(x)(z-x),(z-x)\rangle.
\end{equation}
 Since $u\in C^2(\Omega)$, we have
\[
u(z)-P(z)=o(|z-x|^2)\qquad\textrm{as}\ z\to x,
\]
which means that for every $\eta >0$, there exists $\delta>0$ such that for every $z\in B_\delta(x),$
\begin{equation} \label{u-P.o.pequena}
P(z)-\frac{\eta}{2}|z-x|^2\leq u(z)\leq P(z)+\frac{\eta}{2}|z-x|^2,
\end{equation}
 with equality only when $z=x$. For convenience, let us denote
\[
P_\eta^\pm(z)=P(z)\pm\frac{\eta}{2}|z-x|^2.
\]
Then,
\begin{equation}\label{proof.C2.case.aux.for.remark.in.intro}
\begin{split}
\dashint_{B_{\varepsilon}(0)}
&
\left(
P_\eta^\pm(x+Ay)-P_\eta^\pm(x)
\right)
\,dy 
=
\frac12\dashint_{B_{\varepsilon}(0)}
\left(
\big\langle A^tD^2u(x)Ay,y\big\rangle\pm\eta |Ay|^2
\right)
\,dy
\\
&=
\frac12
\dashint_{B_{\varepsilon}(0)}
\left\langle A^t\left(D^2u(x)\pm\eta I\right)Ay,y\right\rangle
\,dy 
=
\frac{\varepsilon^2}{2(n+2)}\,\textrm{trace}\left(A^t\left(D^2u(x)\pm\eta I\right)A\right),
\end{split}
\end{equation}
by Lemma \ref{lemma.trace.integral}.

On the other hand, since $\cup\mathbb{A}_x$ is bounded there exists $C_x>0$ such that $A\leq C_x I$ for every $A\in\cup\mathbb{A}_x$.
Then, $x+Ay\in B_\delta(x)$ for every $|y|\leq \varepsilon$ and $\varepsilon< \varepsilon_0$, where $\varepsilon_0 C_x\leq \delta$.
Therefore, by \eqref{u-P.o.pequena}, if $\varepsilon<\varepsilon_0$, then
 \begin{equation}\label{C2.case.second.part}
P_\eta^-(x+Ay)
\leq
u(x+Ay)
\leq
P_\eta^+(x+Ay)\qquad\textrm{for every}\ y\in B_\varepsilon.
\end{equation}
Then,
\[
\begin{split}
\frac{\varepsilon^2}{2(n+2)}\,&\textrm{trace}\left(A^t\left(D^2u(x)-\eta I\right)A\right)
 \\
&\leq
\dashint_{B_{\varepsilon}(0)}u(x+Ay)\,dy - u(x)
\leq 
\frac{\varepsilon^2}{2(n+2)}\,\textrm{trace}\left(A^t\left(D^2u(x)+\eta I\right)A\right)
\end{split}
\]
and the result follows by \eqref{eq.hypo}.
\end{proof}

We now prove Theorem \ref{th.sol.viscosa.intro.isaacs}.
%
%
%
%\begin{theorem} \label{th.sol.viscosa.intro} 
%Let $\phi(\varepsilon)$ 
%be a positive function  that satisfies \eqref{hypothesis.phi.intro}.
%Consider $f\in C(\Omega)$ and
%  $F: S^n(\R)\to \R$ defined as in \eqref{eq.thm.main3.intro} that   satisfies \eqref{hyp.intro.semicont}.
%Then, a function $u\in C(\Omega)$ is a viscosity subsolution (respectively, supersolution) of 
%\[
%F(D^2u(x)) =f(x) \quad \mbox{in } \Omega,
%\]
%if and only if
%\[
%u(x)
%\leq
%\mathop{\mathop{\inf}_{A\in\A}}_
%{A\leq \phi(\varepsilon)I}
%\dashint_{B_{\varepsilon}(0)}
%u(x+Ay)
%\,dy
%-
%\frac{\varepsilon^2}{2(n+2)}
%f(x)
%+
%o(\varepsilon^2),
%\qquad\textrm{as $\eps\to0$}
%\]
%(respectively, $\geq$)
% in the viscosity sense.
%\end{theorem}
%

\begin{proof}[Proof of Theorem \ref{th.sol.viscosa.intro.isaacs}]
First, assume that $u$ is a viscosity solution and take 
a $C^2$ function $\phi$ that
touches $u$ at $x \in \Omega$ strictly from below. Then, as $u-\phi$ has a strict minimum at $x$ with $u(x) = \phi(x)$, we have
$$
F\big(x,D^2\phi(x)\big)=\sup_{\mathcal{A}\in\mathbb{A}_x} \inf_{A\in \mathcal{A}}  \tr(A^tD^2\phi(x)A)\leq f(x).
$$
Now, since $\phi$ is $C^2$, Theorem \ref{thm.main1} gives
\[
\begin{split}
&\sup_{\mathcal{A}\in\mathbb{A}_x} \inf_{A\in \mathcal{A}} 
\dashint_{B_{\varepsilon}(0)}
\phi(x+Ay)
\,dy
-
u(x)+
o(\varepsilon^2)  =
\frac{\varepsilon^2}{2(n+2)}
%F\big(x,D^2\phi(x)\big)
%+
%o(\varepsilon^2) \\
\sup_{\mathcal{A}\in\mathbb{A}_x} \inf_{A\in \mathcal{A}}  \tr(A^tD^2\phi(x)A)
\leq 
\frac{\varepsilon^2}{2(n+2)}f(x),
\end{split}
\]
proving that $u$ satisfies 
$$
u(x)
\geq
\sup_{\mathcal{A}\in\mathbb{A}_x} \inf_{A\in \mathcal{A}} 
\dashint_{B_{\varepsilon}(0)}
u(x+Ay)
\,dy
-
\frac{\varepsilon^2}{2(n+2)}f(x)+
o(\varepsilon^2),
$$
in the viscosity sense.  An analogous computation reversing the inequalities shows that when a $C^2$ function $\phi$ 
touches $u$ at $x \in \Omega$ strictly from above, 
$u$ satisfies 
$$
u(x)
\leq
\sup_{\mathcal{A}\in\mathbb{A}_x} \inf_{A\in \mathcal{A}} 
\dashint_{B_{\varepsilon}(0)}
u(x+Ay)
\,dy
-
\frac{\varepsilon^2}{2(n+2)}f(x)+
o(\varepsilon^2)
$$
in the viscosity sense.

Now, assume that the mean-value property holds and take 
a $C^2$ function $\phi$  that 
touches $u$ at $x \in \Omega$ strictly from below. Then, as $u-\phi$ has a strict minimum at $x$ with $u(x) = \phi(x)$, 
 Theorem ~\ref{thm.main1} yields
\[
\frac{\varepsilon^2}{2(n+2)}
F\big(x,D^2\phi(x)\big)
+
o(\varepsilon^2)    
= 
\sup_{\mathcal{A}\in\mathbb{A}_x} \inf_{A\in \mathcal{A}} 
\dashint_{B_{\varepsilon}(0)}
\phi(x+Ay)
\,dy
-
u(x) \leq \frac{\varepsilon^2}{2(n+2)}f(x)+
o(\varepsilon^2).
\]
Dividing by $\varepsilon^2$ and
letting $\varepsilon \to 0$ we get that $u$ is a viscosity supersolution to 
\[
F(x,D^2 u (x))= f(x).
\]
A similar argument reversing the inequalities shows that $u$ is a viscosity subsolution.
\end{proof}

We devote the rest of the section to discuss the examples mentioned in the introduction  in more detail.

\subsection{Isaacs Operators}\label{sec.sup-inf}
As mentioned in the introduction, hypotheses
\eqref{isasc.pablo} and \eqref{assumption.Ax.is.bounded} also cover 
 degenerate operators such as the $k$-th smallest eigenvalue of the Hessian, given by the  Courant--Fischer min-max principle
\begin{equation} \label{Courant-Fisher}
\lambda_{k}\big(D^2u(x)\big)= \max_V \left\{ \min_{v \in V,\ |v|=1}  \langle D^2u(x)v,v\rangle \   : \ V\subset\mathbb{R}^n \text{ subspace of dimension}\ n-k+1 \right\}.
\end{equation}
To write this operator in the form \eqref{isasc.pablo}, let $A_v=v\otimes v$ and note that $A_v^2=A_v$ when $|v|=1$.
 With this notation we have $\tr(A_v^t M A_v)=\langle M v, v\rangle$. Let $G_{n-k+1}$ denote the set of all subspaces $V\subset\mathbb{R}^n$ of dimension $n-k+1$. For each $V\in G_{n-k+1}$, we set $\mathcal{A}_V=\{ v\otimes v \colon v\in V\text{ such that } |v|=1\}$.
Finally we define $\mathbb{A}_k=\{ \mathcal{A}_V\colon V\in G_{n-k+1}\}$, or equivalently,
\[
\displaystyle 
\mathbb{A}_k=\Big\{\big\{A\in S_+^n(\mathbb{R}): \lambda_i(A)=0 \text{ for } i\neq n, \lambda_n (A) =1, \text{ and } v_n\in V\big\} 
 : V\subset\mathbb{R}^n \text{ subspace of dim. } n-k+1\Big\},
\]
where $v_n$ is the eigenvector corresponding to $\lambda_n (A)$.
 We can then write
 $$\lambda_k(M)=
 \sup_{V}\Big\{\inf_{v \in V,\ |v|=1}\langle M v, v\rangle\Big\} 
 = \sup_{\mathcal{A}_V\in \mathbb{A}_k}\inf_{A\in \mathcal{A}_V} \textrm{trace}(A^t M A).
 $$

This  allows us to prove mean-value formulas for \eqref{Courant-Fisher}. In fact, Theorem \ref{thm.main1} gives
\begin{equation}
\label{eq.main.122}
\sup_{\mathcal{A}_V\in \mathbb{A}_k}\inf_{A\in \mathcal{A}_V}
\dashint_{B_{\varepsilon}(0)}
u(x+Ay)
\,dy
-
u(x)
=
\frac{\varepsilon^2}{2(n+2)}
\lambda_k(D^2 u(x))
+
o(\varepsilon^2),
\end{equation}
and the corresponding viscosity analogue following Theorem \ref{th.sol.viscosa.intro.isaacs}.

A variant of \eqref{eq.main.122} is contained in \cite{BR2},
\begin{equation}\label{fromBR2}
\sup_{\dim(V)=n-k+1}\inf_{v\in V, |v|=1} \left\{
\frac{u(x+\eps v)+u(x-\eps v)}{2}   \right\}- u(x)=  \frac{\eps^2}{2} \lambda_k(D^2 u(x))+ o(\eps^2).
\end{equation}

In order to make the connection between both mean-value formulas,  \eqref{eq.main.122} and \eqref{fromBR2}, let 
$A_v=v\otimes v$ with $|v|=1$ as before and
observe that
\[
\begin{split}
2(n+2)\left(
\dashint_{B_{\varepsilon}(0)}
u(x+A_vy)
\,dy- u(x)\right)
&= \eps^2 \langle D^2u(x) v, v\rangle+ o(\eps^2)
\\
&=
u(x+\eps v)+u(x-\eps v)-2u(x)+ o(\eps^2),
\end{split}
\]
where the error estimate is uniform in $v$.

\subsection{Uniformly elliptic fully nonlinear operators} \label{sect-UnifElip}

Let us state the definition of a uniformly elliptic operator for completeness.

\begin{definition}\label{intro.def.elipticidad.unif}
An operator $F:\Omega \times S^n(\R)\rightarrow\mathbb{R}$ is uniformly elliptic with constants $0<\theta\leq\Theta$ if and only if for every $M,N\in S^n(\R)$ with $N\geq0$,
\begin{equation}\label{la.elipticidad.uniforme}
\theta\cdot\textnormal{trace}(N)\leq F(x,M+N)-F(x,M)\leq \Theta\cdot\textnormal{trace}(N)
\end{equation}
for all $x\in\Omega$.
\end{definition}

Here we modify the formulas from the previous section to obtain a mean-value formula for general uniformly elliptic operators.
This is possible because every uniformly elliptic operator can be written as an Isaacs operator.
 We include a proof   for the reader's convenience, see also \cite{CC.2003}.

\begin{lemma}\label{lemma.every.unif.elliptic.is.Isaacs}
	\label{unif-ellip}
 Let 
 $$
 \mathcal{A}_{\theta\Theta}=\left\{A\in S^n(\R):\  \sqrt\theta |\xi|^2\leq \langle A\xi,\xi\rangle \leq \sqrt\Theta |\xi|^2\ \ \forall\xi\in\mathbb{R}^n\right\}.
 $$ 
 An operator $F:\Omega \times S^n(\R)\rightarrow\mathbb{R}$ is uniformly elliptic with constants $0<\theta\leq\Theta$ if and only if
 \begin{equation}\label{Isaacs.characterization}
\begin{split}
F(x,M)&=\underset{N\in S^n(\R)}{\inf}\sup_{A\in\mathcal{A}_{\theta\Theta}}\Big\{\textnormal{trace}\left(A^tMA\right)+F(x,N)-\textnormal{trace}\left(A^tNA\right)\Big\},\\
F(x,M)&=\underset{N\in S^n(\R)}{\sup}\inf_{A\in\mathcal{A}_{\theta\Theta}}\Big\{\textnormal{trace}\left(A^tMA\right)+F(x,N)-\textnormal{trace}\left(A^tNA\right)\Big\}.
\end{split}
 \end{equation}
\end{lemma}

\begin{proof} It is well known that an operator $F:\Omega \times S^n(\R)\to\mathbb{R}$ is uniformly elliptic if and only if,
\[
\inf_{A\in\mathcal{A}_{\theta\Theta}} \textnormal{trace}\left(A^t(M-N)A\right)
\leq F(x,M)-F(x,N)
\leq
\sup_{A\in\mathcal{A}_{\theta\Theta}}\textnormal{trace}\left(A^t(M-N)A\right)
\]
for every $M,N\in S^n(\R)$ and $x\in\mathbb{R}$. Since we have  equalities when  $M=N$, in particular \eqref{Isaacs.characterization} holds.
\end{proof}

With this characterization, we can prove the following.
\begin{theorem}
\label{thm.main2.bb}
Consider  $F:\Omega \times S^n(\R)\to \R$ uniformly elliptic  and 
let $u\in C^2(\Omega)$. Then, for every $x\in\Omega$ we have 
\[
\underset{N\in S^n(\mathbb{R})}{\inf}
\sup_{A\in\A}
\left(
\dashint_{B_{\varepsilon}(0)}
u(x+Ay)
\,dy
+F(x,N)-\textnormal{trace}\left(A^tNA\right)
\right)-
u(x)
=
\frac{\varepsilon^2}{2(n+2)}
\, F\big(x,D^2u(x)\big)
+
o(\varepsilon^2),
\]
as $\eps\to0$,
where $\mathcal{A}=\left\{A\in S^n_+(\R):\  \sqrt\theta |\xi|^2\leq \langle A\xi,\xi\rangle \leq \sqrt\Theta |\xi|^2,\ \ \forall\xi\in\mathbb{R}^n\right\}$.
\end{theorem}

\begin{proof} [Proof of Theorem~\ref{thm.main2.bb}]
Since the family of matrices $\mathcal{A}_{\theta\Theta}$ is uniformly elliptic, in particular it is bounded. 
Then, an analogous result to Lemma~\ref{Lip-cond} can be obtained for the operators considered here and the proof follows 
in the same way as the one of Theorem \ref{thm.main1}. From the expression
$$
F\big(x,D^2u (x)\big)=\underset{N\in S^n(\R)}{\inf}\sup_{A\in\mathcal{A}_{\theta\Theta}}\Big\{\textnormal{trace}\left(A^t D^2u(x)
A\right)+F(x,N)-\textnormal{trace}\left(A^tNA\right)\Big\},
$$
we obtain that for a smooth function $u$ it holds that
\[
\begin{split}
\underset{N\in S^n(\mathbb{R})}{\inf}
\sup_{A\in\A}
\left(
\dashint_{B_{\varepsilon}(0)}
u(x+Ay)
\,dy
+F(x,N)-\textnormal{trace}\left(A^tNA\right)
\right)-
u(x)
=
\frac{\varepsilon^2}{2(n+2)}
\, F\big(x,D^2u(x)\big)
+
o(\varepsilon^2),
\end{split}
\]
as $\eps\to0$,
where $\mathcal{A}=\left\{A\in S^n_+(\R):\  \sqrt\theta |\xi|^2\leq \langle A\xi,\xi\rangle \leq \sqrt\Theta |\xi|^2,\ \ \forall\xi\in\mathbb{R}^n\right\}$.
\end{proof}

\begin{remark} {\rm
Observe that for some operators both Theorems \ref{thm.main1} and \ref{thm.main2.bb} apply.
Even then, the formulas that we get are different.
For example, for Issacs operators of the form
$$
F\big(x,D^2u (x)\big)=\sup_{\alpha\in \mathcal A} \inf_{\beta\in \mathcal B} \textnormal{trace}\left(A_{\alpha\beta}^t
D^2u(x) A_{\alpha\beta}\right),
$$
we get two possible mean-value formulas, namely,
	\[
	\inf_{\alpha\in \mathcal A} \sup_{\beta\in \mathcal B}
	\dashint_{B_{\varepsilon}(0)}
	u(x+A_{\alpha\beta}y)
	\,dy
	-
	u(x)
	=
	\frac{\varepsilon^2}{2(n+2)}
	F\big(x,D^2u (x)\big)
	+
	o(\varepsilon^2),
\]
and
\[		
\begin{split}
 		\underset{N\in S^n(\R)}{\inf}
		\sup_{A\in\A_{\theta\Theta}}
		\left(
		\dashint_{B_{\varepsilon}(0)}
		u(x+Ay)
		\,dy
		+F(x,N)-\textnormal{trace}\left(A^tNA\right)
		\right)&-
		u(x)
		\\
		=
		\frac{\varepsilon^2}{2(n+2)}\,
		&F\big(x,D^2u(x)\big)
		+
		o(\varepsilon^2).
\end{split}		
\]
}
\end{remark}

\subsection{Off-center means and equations involving lower-order terms} \label{sec-lower.order}

When the mean-value formula involves integrals in balls that are not centered at $0$ but at 
$\varepsilon^2 v$ with $|v|=1$, we obtain second-order operators with first-order terms,
\begin{multline}
\label{eq.main.678.34}
\inf_{A\in\A_x}
\dashint_{B_{\varepsilon} (\varepsilon^2 v)}
\frac{u(x+Ay)
-
u(x)}{\varepsilon^2}
\,dy
= 
\inf_{A\in\A_x}\left\{ \frac{1}{2(n+2)} \tr(A^tD^2u(x)A)
+ \langle Du(x), A v \rangle \right\}
+ o(1),
\end{multline}
as $\eps\to0$. Here we assume for simplicity that $\A_x\subset S_+^n(\mathbb{R})$  is  bounded  for each $x\in\Omega$;  
%Finally, let us mention that our results concerning lower-order terms can also be extended to unbounded sets of matrices. In fact, 
 whenever $\A_x$ is not bounded, we
can restrict the argument to matrices that satisfy 
 $A\leq \phi(\varepsilon)I$
 for $\phi(\varepsilon$) satisfying  \eqref{hypothesis.phi.intro}.

The mean-value formula \eqref{eq.main.678.34} is a consequence of the fact that for a $C^2$ function we have
\[
\begin{split}
 \dashint_{B_{\varepsilon} (\varepsilon^2 v)}&
u(x+Ay)
\,dy
-
u(x) 
= \dashint_{B_{\varepsilon} (0)}
u(x+A (z +\varepsilon^2 v))
\, dz
-
u(x) \\
& =
\dashint_{B_{\varepsilon} (0)}
\left(
\big\langle Du(x) , A (z +\varepsilon^2 v) \big\rangle + \frac12 \big\langle D^2u(x) A (z +\varepsilon^2 v) , A (z +\varepsilon^2 v) \big\rangle
\right)
\, dz + o(\varepsilon^2) \\
 & = \varepsilon^2 \langle Du(x) , A  v \rangle +
 \frac12 \dashint_{B_{\varepsilon} (0)}   
  \langle D^2u(x) A z , A z \rangle 
\, dz 
 + o(\varepsilon^2) \\ 
 & = \varepsilon^2 \left\{\langle Du(x) , A  v \rangle + \frac{1}{2(n+2)} \mbox{trace} (A^t D^2u(x) A) \right\}
 + o(\varepsilon^2) .
\end{split}
\]
Notice that the remainder $o(\varepsilon^2)$ is independent of $A$ because  $\A_x$ 
is a  bounded set  for each $x\in\Omega $.

In addition, we observe that when we center the average at 
$\varepsilon^\alpha v$ we obtain: for $\alpha >2$ a pure second-order operator, 
\begin{equation*}
%\label{eq.main.678.99}
\inf_{A\in\A_x}
\dashint_{B_{\varepsilon} (\varepsilon^\alpha v)}
u(x+Ay)
\,dy
-
u(x)
=
\frac{ \varepsilon^2 }{2(n+2)} \inf_{A\in\A_x} \tr(A^tD^2u(x)A)
+ o(\varepsilon^2),
\end{equation*}
as $\eps\to0$; and for $0<\alpha <2$ a pure first-order operator 
\begin{equation*}
%\label{eq.main.678.55}
\inf_{A\in\A_x}
\dashint_{B_{\varepsilon} (\varepsilon^\alpha v)}
u(x+Ay)
\,dy
-
u(x)
= \varepsilon^\alpha
\inf_{A\in\A_x}  \langle Du(x), A v \rangle 
+ o(\varepsilon^\alpha),
\end{equation*}
as $\eps\to0$.

Also, we can look for zero-order terms and consider mean-value properties like
\begin{multline*}
%\label{eq.main.678.44}
(1-\alpha \varepsilon^2)\inf_{A\in\A_x}
 \dashint_{B_{\varepsilon}(0) }
\frac{u(x+Ay)
- u(x)}
{\varepsilon^2}
\,dy
 =  
\frac{1}{2(n+2)}
\inf_{A\in\A_x} \tr(A^tD^2u(x)A)
- \alpha u(x) 
+ o(1),
\end{multline*}
as $\eps\to0$.
Arguing as before, we have
\[
\begin{split}
& 
(1-\alpha \varepsilon^2) \dashint_{B_{\varepsilon}(0) }
u(x+Ay)
\,dy
- u(x) \\
& \qquad =
 \dashint_{B_{\varepsilon}(0) }
u(x+Ay)
\,dy
- u(x) -\alpha \varepsilon^2
 \dashint_{B_{\varepsilon}(0) }
u(x+Ay)
\,dy \\
& \qquad = \varepsilon^2 \Big\{  \frac{1}{2(n+2)} \mbox{trace} (A^t D^2u(x) A) -\alpha u(x) 
\Big\} + o(\varepsilon^2).
\end{split}
\]
Similar mean-value formulas also hold for sup-inf operators with lower-order terms. We leave the details to the reader.

\section{Unbounded operators}\label{sec.unboundd}

In this section we deal with unbounded operators
 $F:S^n(\R)\to \R\cup\{-\infty\}$ given by
\begin{equation}\label{eq.thm.main3.mainsect}
F(M)=\inf_{A\in\A} \tr(A^tMA).
\end{equation}
As mentioned in the introduction,  we restrict the set of matrices where we compute the infimum in the mean-value property by
considering matrices $A\in\A$ such that
$A\leq \phi(\varepsilon)I,$ with $\phi(\eps)$ a positive function satisfying 
\begin{equation}
\label{hypothesis.phi.sect.unbounded}
 \lim_{\varepsilon\to0} \phi(\varepsilon)=+\infty
 \qquad
  \textrm{and}
 \qquad
 \lim_{\varepsilon\to0} \varepsilon\,\phi(\varepsilon) =0.
\end{equation}
The condition $A\leq \phi(\varepsilon)I$ becomes less restrictive as $\varepsilon\to0$, but  is still enough to   make  the mean-value formula \eqref{eq.thm.main3.intro.MVP} local. %and every ellipsoid over which we are integrating is contained in the domain $\Omega$ for $\varepsilon$ small. 
	In fact, for every $x\in\Omega$ and $|y|\leq \varepsilon$,
%	for every $x\in\Omega$, the conditions  $A\leq \phi(\varepsilon)I$  and $|y|\leq \varepsilon$ imply that $x+Ay\in\Omega$ for $\varepsilon$ small enough; 
%  we have
	\[
	\textrm{dist}(x+Ay,x)=|Ay|\leq  \varepsilon\, \phi(\varepsilon)  \leq \textrm{dist}(x,\partial\Omega)
	\]
	for $\varepsilon$ sufficiently small (since $\varepsilon\,\phi(\varepsilon)\to 0$ as $\varepsilon \to 0$).

Recall that we have assumed  \eqref{hyp.intro.semicont} and therefore   $$F \mbox{ is continuous in the cone }\Gamma_\A=\big\{M\in S^{n}(\mathbb{R}): F(M)>-\infty \big\}.$$
Next, we prove that, as long as $\mathcal{A}$ is unbounded, there are matrices $M\in S^{n}(\mathbb{R})$ for which $F(M)=-\infty$.

%\begin{remark}\label{GammaCondition}
%{\rm
%Observe that $F$ is an infimum of linear functions, which are continuous, therefore it is upper semi-continuous.
%Even more, $F$ is concave and continuous in $\Gamma_\mathcal{A}^\circ$.
%That is, to obtain that $F$ is continuous it is enough to require its  lower semi-continuity in $\Gamma_\mathcal{A}\setminus \Gamma_\mathcal{A}^\circ$.
%} 
%\end{remark}

\begin{lemma}
\label{Axbounded} The operator
$$
F(x,M)=\inf_{A\in\A_x} \tr(A^tMA)
$$
is finite for every $M\in  S^n(\R)$ if and only if
$\A_x$ is bounded.
\end{lemma}

\begin{proof} 
It is clear that if $\A_x$ is bounded then $F(x,  M)$ is finite for every $M\in  S^n(\R)$.  To prove the converse, suppose that $F(x,  M)$ is finite for every $M\in  S^n(\R)$ and $\A_x$ is not bounded.
Then there exists a sequence of matrices $A_k\in \A_x$ such that their largest eigenvalues $\lambda_{n} (A_k)$ diverge as $k\to\infty$.
Write $A_k= Q_k^t D_k Q_k$, where  $Q_k$ is an orthogonal matrix and $D_k$ is a diagonal matrix with diagonal entries
$(D_k)_{jj}=\lambda_{j} (A_k)$. Let us assume that the eigenvalues of $A_k$ satisfy $0\le \lambda_{1} (A_k)\le \lambda_{2}
(A_k)\le\ldots\le \lambda_{n} (A_k)$. 
Extracting a subsequence, if needed, we may also assume that $Q_k\to Q_\infty$. 
Set $M=Q_\infty^t J\,  Q_\infty$, where $J$ is the $n\times n$ diagonal matrix with diagonal entries $\{0,0,\ldots,0,-1\}$. 
We then have
\begin{align*}
\tr (A_k^tMA_k) %&= \tr\left( Q_k^t D_k Q_k Q_\infty^t J\,  Q_\infty  Q_k^t D_k Q_k  \right)\\
&= \tr\left(  J Q_\infty Q_k^t D_k^2 Q_k Q_\infty^t       \right) 
= -\left( Q_\infty Q_k^t D_k^2 Q_k Q_\infty^t \right)_{nn}\\
%&= -\sum_{j=1}^n \left(    Q_\infty Q_k^t D_k^2   \right)_{nj} \left(  Q_k Q_\infty^t    \right)_{jn}
% = -\sum_{j=1}^n \left(   \sum_{l=1}^n (Q_\infty Q_k^t)_{nl} (D_k^2)_{lj}\right) \left(  Q_k Q_\infty^t    \right)_{jn}\\
& = -\sum_{j=1}^n   (Q_\infty Q_k^t)_{nj} [\lambda_{j} (A_k)]^2 \left( Q_k Q_\infty^t    \right)_{jn}
= -\sum_{j=1}^n   (Q_\infty Q_k^t)_{nj}^2 [\lambda_{j} (A_k)]^2 \\
&
\le  - (Q_\infty Q_k^t)_{nn}^2 [\lambda_{n} (A_k)]^2,
\end{align*}
which tends to $-\infty$ since $(Q_\infty Q_k^t)_{nn}\to 1$ and $\lambda_{n} (A_k)\to\infty$ as $k\to\infty$ contradicting the fact 
that $F(x,  M)$ is finite. 
\end{proof}

We can now proceed with the proof of Theorem \ref{thm.main3.intro}. 
First, we prove a continuity lemma, analogous to Lemma~\ref{Lip-cond}.
Here, however, we  have the extra restriction 
 $A\leq \phi(\eps)I$ and the continuity of  $F$ from only one side.  This is because  the cone $\Gamma_\A$ is, in principle, neither open nor closed, and when $D^2u(x)\in \partial \Gamma_\A$ we can only use
perturbations of the form $D^2u(x) + \eta I$.

\begin{lemma}
\label{continuity}
For every $M\in \Gamma_\A$, we have
\begin{equation}
\label{eq.hypoUP}
\mathop{\mathop{\inf}_{A\in\A}}_
{A\leq \phi(\varepsilon)I}
\tr(A^t\left(M+\eta I\right)A)
\to
\inf_{A\in\A} \tr (A^tMA)
\qquad
\textrm{as $\eps,\eta\searrow 0$.}
\end{equation}
\end{lemma}

\begin{proof}
Since $\tr(A^t\left(M+\eta I\right)A)\geq \tr(A^tMA)$, we have
\[
\mathop{\mathop{\inf}_{A\in\A}}_
{A\leq \phi(\varepsilon)I}
\tr(A^t\left(M+\eta I\right)A)
\geq
\inf_{A\in\A} \tr (A^tMA).
\]
Let us fix $M\in \Gamma_\A$ and $\delta>0$.
We consider $A_0\in\A$ such that
\[
\inf_{A\in\A} \tr (A^tMA)+\delta/2\geq \tr (A_0^tMA_0).
\]
Since $\phi(\varepsilon)\to\infty$ as $\eps\to 0$ there exists $\eps_0$ such that $A_0\leq \phi(\varepsilon)I$ for every $\eps<\eps_0$.
Let $\eta_0>0$ be such that $\tr(A_0^tA_0)\eta_0<\delta/2$, we get
\[
\begin{split}
\mathop{\mathop{\inf}_{A\in\A}}_{A\leq \phi(\varepsilon)I} &
\tr(A^t\left(M+\eta I\right)A) 
\leq
\tr(A_0^t\left(M+\eta I\right)A_0)\\
&=
\tr(A_0^tMA_0)+\tr(A_0^tA_0)\eta 
 \leq
\inf_{A\in\A} \tr (A^tMA)+\delta/2+\delta/2
\end{split}
\]
for every $\eps<\eps_0$ and $\eta<\eta_0$. We have proved \eqref{eq.hypoUP}.
\end{proof}

%These perturbations allow us to obtain an upper bound with an argument similar to the one used before. To obtain the lower bound we need a different argument and at this point is where we use that $F$ is continuous in $\Gamma_\A$.

Lemma \ref{continuity}  allows us to obtain an upper bound for
\begin{equation}\label{what.needs.to.be.bounded}
\mathop{\mathop{\inf}_{A\in\A}}_
{A\leq \phi(\varepsilon)I}
\dashint_{B_{\varepsilon}(0)}
u(x+Ay)
\,dy
-
u(x).
\end{equation}
To conclude the lower bound we need a different argument and it is at this point that we use that $F$ is continuous in $\Gamma_\A$. The continuity of $F$ is a necessary condition
in order to have a mean-value property, see Example~\ref{remark.ejemplo2} below.

\begin{proof}[Proof of Theorem \ref{thm.main3.intro}]
From \eqref{hypothesis.phi.sect.unbounded} it follows that  $A\leq \phi( \varepsilon)I$ and $|y|\leq \varepsilon$ imply
 $x+Ay\in B_\delta(x)$  for $\varepsilon< \varepsilon_0$, where $\varepsilon\, \phi(\varepsilon)< \delta$ for every $\varepsilon<\varepsilon_0$ and $\delta$ is as in \eqref{u-P.o.pequena}. Therefore, by \eqref{u-P.o.pequena}, if $\varepsilon<\varepsilon_0$, then
 \begin{equation}\label{C2.case.second.part.78}
u(x+Ay)
\leq
P_\eta^+(x+Ay)\qquad\textrm{for every}\ y\in B_\varepsilon(0),
\end{equation}
for
\[
P_\eta^+(z)=P(z)+\frac{\eta}{2}|z-x|^2
\]
with $P$ given by \eqref{definition.P.main}.
Then, following the proof of Theorem \ref{thm.main1}, we have
\begin{equation}
\label{lower-upperBound}
\begin{split}
\dashint_{B_{\varepsilon}(0)}u(x+Ay)\,dy - u(x)
\leq 
\frac{\varepsilon^2}{2(n+2)}\,\textrm{trace}\left(A^t\left(D^2u(x)+\eta I\right)A\right).
\end{split}
\end{equation}
By \eqref{eq.hypoUP} this gives us an upper bound of \eqref{what.needs.to.be.bounded}.

To obtain a lower bound we use the continuity of $F$.
Given $x$ and $\eta>0$ there exists $\delta>0$ such that
\[
F(D^2 u(z))\geq  F(D^2 u(x))-\eta
\]
for every $z\in B_\delta(x)$.
Then, there exists $\eps_0>0$ such that for every $\eps<\eps_0$, $y\in B_{\eps}(0)$ and $A\in\A$ with $A\leq \phi( \varepsilon)I$ we have $x+Ay\in B_\delta(x)$.
Let us fix $A\in\A$. We have
\[
\trace(A^t D^2 u(x+Ay) A)\geq F(D^2 u(x+Ay))\geq F(D^2 u(x))-\eta.
\]
We consider $v(y)=u(x+Ay)$, and we observe that $$\Delta v (y) =\trace(A^t D^2 u(x+Ay) A)\geq F(D^2 u(x))-\eta .$$ 
Therefore, from the mean-value formula for the Laplacian, we get
\[
\frac{\eps^2}{2(n+2)}\big(F(D^2 u(x))-\eta\big)
\leq
\dashint_{B_{\varepsilon}(0)}v(y)\,dy - v(0)
=
\dashint_{B_{\varepsilon}(0)}u(x+Ay)\,dy - u(x).
\]
With this lower bound, taking infimums and the limit as $\eta \to 0$, we have completed the proof.
\end{proof}

\begin{remark} {\rm
Our assumptions imply that $F(D^2u(x))>-\infty$ (since $u$ is assumed to be $\mathcal{A}$-admissible).
Observe that in the case that $F(D^2u(x))=-\infty$ by the upper bound in \eqref{lower-upperBound} we also get
\[
\inf_{A\in\A}
\dashint_{B_{\varepsilon}(0)}u(x+Ay)\,dy - u(x)
=-\infty.
\]
}
\end{remark}

Observe that condition \eqref{hyp.intro.cont} is necessary. In Example  \ref{remark.ejemplo2} this assumption is not satisfied and
the mean-value formula fails.

\begin{remark}\label{remark.F.geq.C}
{\rm
In some important examples we have
\begin{equation}
\label{hypo.F=0}
F(M)=0 \quad \textrm{for every}\  M\in \partial {\Gamma}_\A
\end{equation}
and
\[
{\Gamma}_\A=\Big\{M\in S^{n}(\mathbb{R}): \tr(A^t M A)\geq 0\ \textrm{for all}\ A\in \A \Big\}.
\]
This holds for the Monge-Amp\`ere equation, see \eqref{mongeasinf}, and more generally for the $k$-Hessians, see Lemma \ref{def.k} and Remark \ref{remark.k-hess.zero.boundary} below.
Condition \eqref{hypo.F=0} was used in \cite{HL1} 
in relation to existence and uniqueness of solutions 
to general fully nonlinear second-order PDEs. 
We observe that condition \eqref{hypo.F=0} implies the continuity of $F$ in $\Gamma_\A$ and therefore Theorem \ref{thm.main3.intro} can be proved assuming 
 \eqref{hypo.F=0}  instead of \eqref{hyp.intro.cont}.
In fact, if there exists $C\in\R$ such that the cone $\Gamma_\A =\{ M: F(M) > -\infty\}$ can be written as
\[
\Gamma_\A=\big\{M:F(M)\geq C\big\}
\]
and $F\equiv C$ in $\partial\Gamma_\A$, then $F$ is continuous; that is, \eqref{hyp.intro.cont} holds.
This can be easily deduced from the fact that $F$ is lower 
semi-continuous on $\partial\Gamma_\A$  because $F$ attains its minimum at every point of the boundary.}
\end{remark}

In the following example we show that  condition \eqref{hypo.F=0} is not necessary by exhibiting
an operator that does not satisfy \eqref{hypo.F=0} for which the mean-value property holds.
Here,
$F$ is not constant on the boundary of the associated cone; nevertheless Theorem \ref{thm.main3.intro} applies
since $F$ is continuous.

\begin{example} \label{remark.ejemplo1} {\rm
If we consider 
\[
\A=\left\{
\begin{bmatrix}
1&0\\
0& n 
\end{bmatrix}
:
n=0,1,2,\ldots\right\},
\]
then we have
\[
F(D^2u )=\inf_{A\in\A} \tr(A^tD^2u A)
=
\begin{cases}
u_{x_1x_1} & \text{ if } u_{x_2x_2}\geq 0\\
-\infty & \text{ if } u_{x_2x_2}<0.
\end{cases}
\]
The equation $F(D^2u)=f$ is equivalent to have $u_{x_1x_1}=f$ and $u_{x_2x_2}\geq 0$. 
It is not a nice equation, in the sense that it is overdetermined. 
Observe, however,  that $F$ is continuous where it is finite and therefore our result applies to this case and we have a mean
value formula for this operator.

Associated with this $F$, we have the closed cone $$\Gamma=\Big\{M:M_{22}\geq 0 \Big\} = \Big\{ M : F(M) > -\infty \Big\}.$$
It is interesting that in this case $F$ is not constant in the boundary of $\Gamma$.
Observe that both 
\[
\begin{bmatrix}
0& 0\\
0& 0
\end{bmatrix}
\quad
\text{and}
\quad
\begin{bmatrix}
1& 0\\
0& 0
\end{bmatrix}
\]
belong to $\Gamma$.
Even more, they belong to the boundary of $\Gamma$, in fact 
\[
\begin{bmatrix}
0& 0\\
0& -\delta
\end{bmatrix}
\quad
\text{and}
\quad
\begin{bmatrix}
1& 0\\
0& -\delta
\end{bmatrix}
\]
are not in $\Gamma$ for every $\delta>0$.
Hence, $F$ is not constant on the boundary of the cone since we have 
\[
F\left(
\begin{bmatrix}
0& 0\\
0& 0
\end{bmatrix}
\right)
=0
\quad
\text{and}
\quad
F\left(
\begin{bmatrix}
1& 0\\
0& 0
\end{bmatrix}
\right)
=1.
\]
}
\end{example}

The second example shows that continuity of $F$ is necessary for the validity
of a mean-value formula.

\begin{example}
\label{remark.ejemplo2} {\rm
We provide an example showing that the continuity of $F$ is a necessary condition for Theorem~\ref{thm.main3.intro}.
We consider the set $\A=\{A_
\delta\}_{\delta>0}$ where
\[
A_\delta=
\begin{bmatrix}
 \sqrt{2}/\sqrt{\delta} & 0\\
0 &  1/\delta
\end{bmatrix}
\]
and define
\[
F(D^2 u)
=
\inf_{A\in\A}
\trace(A^t D^2u A)
=\inf_{\delta>0}\left( \frac{2}{\delta}\, u_{x_1x_1}+\frac{1}{\delta^2} \,u_{x_2x_2}\right).
\]
We have
\[
F(D^2 u)=
\left\{
\begin{aligned}
-&\infty & &\text{if } u_{x_2x_2}<0\\
-&\infty & &\text{if } u_{x_2x_2}=0 \text{ and } u_{x_1x_1}<0\\
&0 & &\text{if } u_{x_2x_2} \geq 0 \text{ and } u_{x_1x_1}\geq 0\\
 -& \frac{u_{x_1x_1}}{u_{x_2x_2}} & &\text{if } u_{x_2x_2}> 0 \text{ and }u_{x_1x_1}<0,
\end{aligned}
\right.
\]
see Figure \ref{fig:F}.
Observe that $F$ is not continuous at the origin, even if we restrict the domain to the set where it is finite.
\begin{figure}
\begin{center}
\begin{tikzpicture}
\node[rectangle,text=red,fill = red!30,minimum width = 5cm, minimum height = 1.5cm] 
 at (0,-0.75){$F=-\infty$};
 \node[rectangle,text=blue,fill = blue!30,minimum width = 2.5cm, minimum height = 2.25cm] 
 at (1.25,1.14){$F=0$};
\node [text=black] at (-0.6,2.75){$u_{x_2x_2}$};
\node [text=black] at (3.2,-0.4){$u_{x_1x_1}$};
\node [text=black] at (-1.4,1.13){$\displaystyle F= - \frac{u_{x_1x_1}^2}{u_{x_2x_2}}$};
\draw [<-,thick,red] (-3,0) -- (0,0);
\draw [->,thick,blue] (0,0) -- (3.25,0);
\draw [->,thick,blue] (0,0) -- (0,3);
\node at (0,0)[circle,fill=blue,inner sep=1.5pt]{};
\end{tikzpicture}
\end{center}
\caption{The operator $F$ in Example \ref{remark.ejemplo2}}
\label{fig:F}
\end{figure}
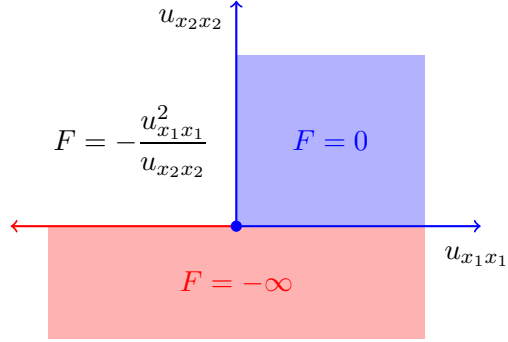

Consider 
\[
u(x_1,x_2)=-|x_1|^{5/2}+x_1^2x_2^2+x_2^{10}.
\]
Observe that $u\in C^2$ and $\mathcal{A}$-admissible.
In fact we have $$u_{x_2x_2} (x_1,x_2)=2x_1^2+90x_2^{8}\geq 0,$$
with equality only at the origin, with $u_{x_1x_1} (0,0)=0$ (notice that $D^2u(0, 0) = 0$). 
Therefore, we have that $F(D^2u (x_1,x_2)) > -\infty$ for every $(x_1,x_2) \in \mathbb{R}^2$. 

Let us show that the mean-value formula in Theorem \ref{thm.main3.intro} is not satisfied at $(0,0)$, i.e.,
$$
\inf_{\delta>0}\vint_{B_{\eps}(0)} u(A_\delta y) \,dy
-
u(0,0)
\neq
\frac{\varepsilon^2}{2(n+2)}
F(D^2u (0,0))
+
o(\varepsilon^2).
$$
Since $D^2 u(0,0)=0$, we have $F(D^2 u(0,0))=0$
and we only need to prove that
\begin{equation}\label{what.we.want.example.2}
\inf_{\delta>0}\vint_{B_{\eps}(0)} u(A_\delta y) \,dy \neq o(\eps^2). 
\end{equation}
An explicit computation shows that
\[
\begin{split}
\vint_{B_{\eps}(0)} u(A_\delta y) \,dy
=
\vint_{B_{\eps}(0)} \left(\frac{-2^{5/4}|y_1|^{5/2}}{\delta^{5/4}}+\frac{2y_1^2y_2^2}{\delta^3}+\frac{y_2^{10}}{\delta^{10}}\right) \,dy
=
-C_1\frac{\eps^{5/2}}{\delta^{5/4}}
+C_2\frac{\eps^{4}}{\delta^3}
+C_3\frac{\eps^{10}}{\delta^{10}}.
\\
\end{split}
\]
For $\delta=\eps^{1/2}$ we have
\[
\frac{\eps^{5/2}}{\delta^{5/4}}=\eps^{15/8},\quad
\frac{\eps^{4}}{\delta^3}=\eps^{5/2} \quad \text{and} \quad
\frac{\eps^{10}}{\delta^{10}}=\eps^{5}
\]
and we obtain
\[
\inf_{\delta>0}\vint_{B_{\eps}(0)} u(A_\delta y) \,dy \leq-C \eps^{15/8},
\]
which proves \eqref{what.we.want.example.2}.
Therefore the mean-value formula in Theorem \ref{thm.main3.intro}
does not hold in this case. 
}
\end{example}

\section{$k$-Hessian operators} \label{section.k.hess}

 A relevant example where  Theorem~\ref{thm.main3.intro}  applies are the $k$-Hessian operators. For $k=2,\ldots,n$ the $k$-Hessian operators are given by  elementary symmetric polynomials in the eigenvalues of the Hessian, i.e.,
\[
\sigma_k(\lambda_1,\dots,\lambda_n)=\sum_{1\leq i_1< i_2<\cdots<i_k\leq n}\lambda_{i_1}\lambda_{i_2}\dots\lambda_{i_k}.
\]
We write the operators  in the form
\begin{equation}\label{k.hessian.rewritten}
 F_k(D^2u (x))= k\, \sigma_k(\lambda(D^2u (x)))^{\frac{1}{k}},
\end{equation}
where $\lambda(D^2u (x))=\big(\lambda_1(D^2u(x)),\ldots,\lambda_n(D^2u(x))\big)$, see Lemma \ref{def.k2} below.

%We have written the operator in the form \eqref{k.hessian.rewritten} to fit our framework, but it is clear that our result applies to solutions of $\sigma_k(\lambda(D^2u))=f\geq 0$ writing this equation as $k\, \sigma_k(\lambda(D^2u))^{\frac{1}{k}} = k f^{\frac{1}{k}} = \tilde{f}$.
%We assume that $k>1$ as the Laplacian behaves in a different way and the mean-value formula for harmonic functions is well known.

We recall some definitions and properties   of elementary symmetric polynomials, see for example \cite{Wang}.
We define the cone 
\[
\Gamma_k=\Big\{\lambda\in\R^n: \sigma_j(\lambda)> 0\ \text{for all}\ j=1,\dots,k \Big\}.
\]
With a slight  abuse of notation we write $M\in\Gamma_k$ to denote that $\lambda(M)\in \Gamma_k$.
We have
\[
\overline\Gamma_k=\Big\{\lambda\in\R^n: \sigma_j(\lambda)\geq  0\ \text{for all}\ j=1,\dots,k \Big\}
\]
and $\overline\Gamma_k^\circ=\Gamma_k$.
Let us define
\[
\sigma_{k-1,i}(\gamma_1,\dots,\gamma_n)=\sigma_{k-1}(\gamma_1,\dots,\gamma_{i-1},0,\gamma_{i+1},\dots,\gamma_n)
\]
and 
\[
\A_k=
\Big\{
A: \ \lambda_i^2(A)=\sigma_{k-1,i}(\gamma) \text{ with } \gamma=(\gamma_1,\dots,\gamma_n)\in  \Gamma_k \text{ and } \sigma_k(\gamma)=1
\Big\}.
\]
Then, we have $\A_k\subset S^n_+(\mathbb{R})$ since
\[
\sigma_{k-1,i}(\gamma )>0,\qquad\forall \gamma \in \Gamma_k,
\]
see \cite{Wang}.
Also observe that by the continuity of $\sigma_{k-1,i}(\gamma)$ we get $\sigma_{k-1,i}(\gamma)\geq 0$ for every $\gamma\in \overline \Gamma_k$.

Our goal is find the $k$-Hessian counterpart of formula \eqref{mongeasinf}, which holds for the Monge-Amp\`ere operator ($k=n$). We do that in two steps, first show the result for numbers, then for matrices.
\begin{lemma}
\label{def.k}
For every $\mu=(\mu_1,\ldots,\mu_n)\in\R^n$ we have
\[
\mathop{\mathop{\inf}_{\gamma\in  \Gamma_k}}_
{\sigma_k(\gamma)=1}
\sum_{i=1}^n \mu_i\, \sigma_{k-1,i}(\gamma)
=
\left\{
\begin{aligned}
&k\, \sigma_k(\mu)^{\frac{1}{k}}
 & &\textrm{if}\ \mu\in \overline \Gamma_k,\\
& -\infty& &\textrm{otherwise}.
\end{aligned}
\right.
\]
\end{lemma}

\begin{remark}\label{remark.k-hess.zero.boundary} {\rm
Note that 
\[
\mathop{\mathop{\inf}_{\gamma\in  \Gamma_k}}_
{\sigma_k(\gamma)=1}
\sum_{i=1}^n \mu_i\, \sigma_{k-1,i}(\gamma)
=
0 \qquad\textrm{if}\ \mu\in \partial \Gamma_k,
\]
since 
\begin{equation}\label{boundary.gamma.k}
\partial\Gamma_k=\Big\{M\in S^n (\mathbb{R}): \sigma_k(\lambda(M))=  0\ \textrm{and}\ \sigma_j(\lambda(M))\geq  0\ \text{for all}\ j=1,\dots,k-1   \Big\}.
\end{equation}
%\begin{equation}\label{boundary.gamma.k}
%\partial\Gamma_k=\Big\{M\in S^n (\mathbb{R}): \sigma_k(\lambda(M))=  0\ \textrm{and}\ \sigma_j(\lambda(M))>  0\ \text{for all}\ j=1,\dots,k-1   \Big\}\cup\{0\}.
%\end{equation}
}
\end{remark}

\begin{proof}[Proof of Lemma \ref{def.k}]
From formula (1.3) in \cite{Ca-Ni-Sp2} (see also (xi) in \cite{Wang}), we have that for all $\mu,\gamma\in \Gamma_k$, 
\[
k\, \sigma_k(\mu)^{\frac{1}{k}}\,\sigma_k(\gamma)^{\frac{k-1}{k}}
\leq 
\sum_{i=1}^n \mu_i \,\sigma_{k-1,i}(\gamma),
\]
with an equality for $\gamma^*= \sigma_k(\mu)^{-\frac{1}{k}}\,\mu$, i.e.,
\[
\sum_{i=1}^n \mu_i \,\sigma_{k-1,i}(\gamma^*)
=\sum_{i=1}^n \mu_i \,\frac{\sigma_{k-1,i}(\mu)}{\sigma_k(\mu)^{\frac{k-1}{k}}}
= \frac{k\,\sigma_k(\mu)}{\sigma_k(\mu)^{\frac{k-1}{k}}}
=k\, \sigma_k(\mu)^{\frac{1}{k}}.
\]

For $\mu\in\partial\Gamma_k$, observe that $(\mu_1+\eps,\dots,\mu_n+\eps)\in\Gamma_k$ for every $\eps>0$.
We consider 
\[
\gamma_i^\eps=(\mu_i+\eps)\, \sigma_k(\mu_1+\eps,\dots,\mu_n+\eps)^{-\frac{1}{k}}
\]
and we have
\[
k\, \sigma_k(\mu_1+\eps,\dots,\mu_n+\eps)^{\frac{1}{k}}
=
\sum_{i=1}^n (\mu_i+\eps)\, \sigma_{k-1,i}(\gamma^\eps)
\geq
\sum_{i=1}^n \mu_i\, \sigma_{k-1,i}(\gamma^\eps),
\]
 since $\sigma_{k-1,i}(\gamma^\eps)\geq 0$. 
To conclude observe that $$k\, \sigma_k(\mu_1+\eps,\dots,\mu_n+\eps)^{\frac{1}{k}} \to k\, \sigma_k(\mu)^{\frac{1}{k}}$$ as $\varepsilon\to0$.

The only remaining case is when $\mu\not\in\overline\Gamma_k$.
Notice that if $a>0$ is large enough such that $\mu_i+a>0$ for every $i\in\{1,\dots,n\}$, we get $(\mu_1+a,\dots,\mu_n+a)\in\Gamma_k$.
Thus, there exists $a>0$ such that $(\mu_1+a,\dots,\mu_n+a)\in\partial\Gamma_k$. Observe that  $\sigma_k(\mu_1+a,\dots,\mu_n+a)=0$
by \eqref{boundary.gamma.k}. 
Let us fix one such  value $a$ and  consider 
\[
\gamma_i^\eps=(\mu_i+a+\eps)\, \sigma_k(\mu_1+a+\eps,\dots,\mu_n+a+\eps)^{-\frac{1}{k}}.
\]
As before, we have
\[
\sum_{i=1}^n (\mu_i+a+\eps)\, \sigma_{k-1,i}(\gamma^\eps)
=
k\, \sigma_k(\mu_1+a+\eps,\dots,\mu_n+a+\eps)^{\frac{1}{k}}
\]
and 
therefore,
\[
\begin{split}
\lim_{\eps\to 0}&\sum_{i=1}^n \mu_i \,\sigma_{k-1,i}(\gamma^\eps)
=-\lim_{\eps\to 0}(a+\varepsilon)\sum_{i=1}^n \sigma_{k-1,i}(\gamma^\eps)
+k\, \lim_{\eps\to 0}\sigma_k(\mu_1+a+\eps,\dots,\mu_n+a+\eps)^{\frac{1}{k}}
\\
&= -a(n-k+1)\lim_{\eps\to 0} \sigma_{k-1}(\gamma^\eps)
= -a(n-k+1)\lim_{\eps\to 0} \left(\frac{\sigma_{k-1}(\mu_1+a+\eps,\dots,\mu_n+a+\eps)}{\sigma_{k}(\mu_1+a+\eps,\dots,\mu_n+a+\eps)^\frac{k-1}{k}}
\right).
\end{split}
\]

Recall that $$\lim_{\eps\to 0} \sigma_{k}(\mu_1+a+\eps,\dots,\mu_n+a+\eps)= \sigma_{k}(\mu_1+a,\dots,\mu_n+a)=0,$$ 
and if we have $$\lim_{\eps\to 0} \sigma_{k-1}(\mu_1+a+\eps,\dots,\mu_n+a+\eps)= \sigma_{k-1}(\mu_1+a,\dots,\mu_n+a)>0$$ we can 
conclude that $$ \lim_{\eps\to 0}\sum_{i=1}^n \mu_i \,\sigma_{k-1,i}(\gamma^\eps)=-\infty $$
as desired.
If this is not the case we have to look at the rate of convergence in more detail.

We consider two cases. First we assume that  $\mu_i+a\neq 0$ for some $i\in\{1,\ldots, n\}$.
In this case
we have that $\sigma_{j}(\mu_1+a,\dots,\mu_n+a)>0$ for some $j\in\{1,\ldots, k-1\}$
because otherwise
\[
\sum_{i=1}^n (\mu_n+a)^2= 
\sigma_{1}^2(\mu_1+a,\dots,\mu_n+a)
-2 \sigma_{2}(\mu_1+a,\dots,\mu_n+a)
=0,
\]
a contradiction.
%Then, since we are assuming that $\mu_i+a\neq 0$ for every $i$,  
Let $l\in\{1,\ldots, k-1\}$ be the largest integer such that $\sigma_{l}(\mu_1+a,\dots,\mu_n+a)>0$.

We have
\[
\begin{split}
 \sigma_{k} & (\mu_1+a+\eps,\dots,\mu_n+a+\eps)
=
\sum_{i=0}^{k}
\binom{n-i}{n-k}\,
\sigma_{i}(\mu_1+a,\dots,\mu_n+a)\,\varepsilon^{k-i},
\\
&=
\varepsilon^{k-l}
\left[
\binom{n}{n-k}\,\varepsilon^{l}
+
\binom{n-1}{n-k}\,
\sigma_{1}(\mu_1+a,\dots,\mu_n+a)\,\varepsilon^{l-1}
+
\cdots \right. \\
& \qquad\qquad\qquad \qquad\qquad\qquad \qquad\qquad  \left.  \cdots +
\binom{n-l}{n-k}\,
\sigma_{l}(\mu_1+a,\dots,\mu_n+a)
\right]
\\
&=
\varepsilon^{k-l}
\left[
\binom{n-l}{n-k}\,
\sigma_{l}(\mu_1+a,\dots,\mu_n+a)
+O(\varepsilon)
\right]
\end{split}
\]
as $\eps\to 0$ (in the first equality we have defined $\sigma_0(\mu_1+a,\dots,\mu_n+a)=1$). Similarly,
\[
\sigma_{k-1}(\mu_1+a+\eps,\dots,\mu_n+a+\eps)
=
\varepsilon^{k-1-l}
\left[
\binom{n-l}{n-k+1}\,
\sigma_{l}(\mu_1+a,\dots,\mu_n+a)
+O(\varepsilon)
\right]
\]
as $\eps\to 0$. 
Therefore,
\[
\frac{\sigma_{k-1}(\mu_1+a+\eps,\dots,\mu_n+a+\eps)}{\sigma_{k}(\mu_1+a+\eps,\dots,\mu_n+a+\eps)^\frac{k-1}{k}}
=
\left(
\frac{
\binom{n-l}{n-k+1}\,
\sigma_{l}(\mu_1+a,\dots,\mu_n+a)
+O(\varepsilon)
}
{
\left[
\binom{n-l}{n-k}\,
\sigma_{l}(\mu_1+a,\dots,\mu_n+a)
+O(\varepsilon)
\right]^\frac{k-1}{k}
}
\right)
\,
\varepsilon^{-\frac{l}{k}}
\]
as $\eps\to 0$ 
and we conclude that
\[
\lim_{\eps\to 0}\sum_{i=1}^n \mu_i \,\sigma_{k-1,i}(\gamma^\eps)=-\infty.
\]

Finally, in the case that  $\mu_i=-a$ for every $i\in\{1,\dots,n\}$ we consider $\gamma_b=(1,\dots,1,b,1/b,0,\dots,0)$ where $b>0$ and $k-2$ coordinates are equal to 1. We have $\gamma_b\in\Gamma_k$, $\sigma_k(\gamma_b)=1$ and
\[
\sum_{i=1}^n \mu_i \,\sigma_{k-1,i}(\gamma_b)=
-a \sum_{i=1}^n \sigma_{k-1,i}(\gamma_b)=
 -a(n-k+1) \sigma_{k-1}(\gamma_b)
=
-a (n-k+1)\Big(k-2+b+\frac{1}{b}\Big)
\]
which goes to $-\infty$ as $b\to\infty$.
\end{proof}

We are now ready to show Lemma \ref{def.k2}, the matrix counterpart of Lemma \ref{def.k}.
\begin{proof}[Proof of Lemma \ref{def.k2}]
Given $M\in S^n(\R)$ we consider $A\in\A_k$ such that both matrices are diagonal in the same basis. We have
\[
\tr(A^tMA)
=\tr(AA^tM)
=\sum_{i=1}^n \lambda_{i}(AA^t)\lambda_{i}(M)
=\sum_{i=1}^n \lambda_{i}^2(A)\lambda_{i}(M).
\]

Then, by the definition of $\A_k$ and Lemma~\ref{def.k} we get that
\[
\inf_{A\in\mathcal{A}_k}\tr(A^tMA)
\leq 
\mathop{\mathop{\inf}_{\gamma\in  \Gamma_k}}_
{\sigma_k(\gamma)=1}
\sum_{i=1}^n \lambda_{i}(M) \,\sigma_{k-1,i}(\gamma)
=
\left\{
\begin{aligned}
&k\, \sigma_k(\lambda(M))^{\frac{1}{k}}
 & &\textrm{if}\ M\in \overline \Gamma_k,\\
& -\infty& &\textrm{otherwise}.
\end{aligned}
\right.
\]
Therefore,  it only remains to prove that 
\begin{equation}\label{interm.goall}
\tr(A^tMA)\geq k\, \sigma_k(\lambda(M))^{\frac{1}{k}}
\end{equation}
for every $A\in\mathcal{A}_k$ and $M\in \overline \Gamma_k$.
To that end we recall the following inequality by Marcus (see \cite{Marcus})
\[
\tr(XM)\geq \min_{p}\sum_{i=1}^n \lambda_i(M)\lambda_{p(i)}(X),
\]
where $p$ ranges over the permutations of the numbers $\{1,\dots,n\}$.
Recalling that $A\in\mathcal{A}_k$ we obtain
\[
\tr(A^tMA)\geq \min_{p}\sum_{i=1}^n \lambda_i(M)\sigma_{k-1,p(i)}(\gamma)
\]
for some $\gamma\in \Gamma_k$ such that $\sigma_k(\gamma)=1$.
Let $\tilde p$ be the permutation where the minimum is attained and
observe that 
\[
\sigma_{k-1,\tilde{p}(i)}(\gamma)=\sigma_{k-1,i}(\tilde\gamma)
\]
where $\tilde\gamma$ is such that $\tilde\gamma_{i}=\gamma_{\tilde{p}(i)}$.
We have
\[
\tr(A^tMA)\geq \sum_{i=1}^n \lambda_i(M)\sigma_{k-1,i}(\tilde\gamma).
\]
Observe that $\tilde \gamma\in\Gamma_k$ and $\sigma_k(\tilde\gamma)=1$, hence, by Lemma~\ref{def.k} we have
\[
\sum_{i=1}^n \lambda_i(M)\sigma_{k-1,i}(\tilde\gamma)\geq 
k\,\sigma_k(\lambda(M))^{\frac{1}{k}}
\]
and \eqref{interm.goall} follows.
\end{proof}

Recall that a function $u\in C^2(\Omega)$ is called $k$-convex 
whenever  $D^2u(x)\in
\overline\Gamma_k$ for every $x\in\Omega$, see \cite{Trudinger.Wang.1997,Trudinger.Wang.1999,Trudinger.Wang.2002}. 
As a consequence of  Lemma \ref{def.k2} we  have the following result.

\begin{corollary}\label{coro.GammaA=Gammak}
A function $u\in C^2(\Omega)$ is $k$-convex if and only if it is $\A_k$-admissible. 
In other words, we have
\[%\begin{equation}\label{corollary.k.convex}
 \Gamma_{\mathcal{A}_k}=\overline{\Gamma}_k. 
\]%\end{equation}
\end{corollary}

As a consequence of the Lemma \ref{def.k2} we can obtain Theorem~\ref{thm.k-hessians} from Theorem \ref{thm.main3.intro}.

\begin{proof} [Proof of Theorem~\ref{thm.k-hessians}]
From Corollary \ref{coro.GammaA=Gammak},  we get
$\Gamma_{\mathcal{A}_k}=\overline{\Gamma}_k,$
i.e., it is equivalent to be $k$-convex and $\mathcal{A}_k$-admissible.
Also observe that $F_k$ is continuous in $\A_k$.
Therefore we are under the hypothesis of Theorem~\ref{thm.main3.intro}
and the result follows.
\end{proof}

\section{Examples in the Heisenberg Group}\label{section.heisenberg}

Let $\mathbb{H}=(\mathbb{R}^3, *)$ be the Heisenberg group. We write a point  $q\in\mathbb{H}$  as $q=(x,y,z)$. The point $\bar{q}=(x,y)$ is the horizontal projection of $q$. 
When convenient,  we will also use the notation $q=(q_1,q_2,q_3) = (\bar{q}, q_3)$. The  group operation is 
$$q*q' = (x,y,z)*(x',y',z')=\Big(x+x', y+y', z+z'+\frac{1}{2}(xy'-yx')\Big).$$ 
The Kor\'{a}niy gauge  is given by
 $$|q|_K= \big((x^2+y^2)^2+16z^2\big)^{1/4}.$$ It induces a left-invariant metric $d(q,q')=|q^{-1}*q'|_K$ in $\mathbb{H}$. 
We also have a family of  anisotropic dilations: 
$$\rho_\lambda(x,y,z)=(\lambda x, \lambda y, \lambda^2 z), \qquad \quad {\lambda>0}$$ that are group homomorphisms.
The Kor\'{a}niy gauge  and the  Kor\'{a}niy metric are homogeneous with respect to the dilations
$$|\rho_\lambda(q)|_K= \lambda |q|_K, \qquad d(\rho_\lambda(q), \rho_\lambda(q'))= \lambda \,d(q,q').$$
The open ball centered at $q$ with radius $r>0$  is a translation and dilation of the open ball centered at $0$ of radius $1$
$$B_r(q) = \{q'\in \mathbb{H}; ~ d(q,q')<r\} = q*B_r(0)= q*\rho_r (B_1(0)).$$
 Euclidean balls in $2$ dimensions will be denoted by $B^2_r$. \par

The vector fields
$$X=\partial_{x}-\frac{y}{2}\partial_{z},\quad Y=\partial_{y}+\frac{x}{2}\partial_{z},\quad
Z=\partial_{z} $$ are left-invariant and form a basis for the Lie algebra of $\mathbb{H}$. 
The only non-trivial commuting relation is $Z=[X,Y]$.  The horizontal tangent space at the point $q$ is the plane generated by $X(q)$ and $Y(q)$ 
\begin{equation}\label{horizontalplane}
T_q= \textrm{span}\left\{\left(1, 0, -\frac{y}{2}\right), \left(0,1,
\frac{x}{2}\right)\right\} = \left(\frac{y}{2}, - \frac{x}{2},1\right)^\perp.\end{equation}
The horizontal gradient and the sub-Laplacian of
a function $v:\mathbb{H}\to\mathbb{R}$ are respectively the vector field and function
$$ \nabla_{\mathbb{H}} v= (Xv) X+ (Yv)Y, \qquad   \Delta_{\mathbb{H}} v=(X^2+Y^2)v.$$
The  horizontal second derivatives are given by the non necessarily symmetric $2\times2$ matrix 
$$\nabla_{\mathbb{H}}^2v(q)= \left(
\begin{array}{cc}
X^2v(q) & XYv(q) \\
YXv(q) & Y^2 v(q)
\end{array}
\right).
$$
The symmetrized horizontal second derivatives are $(\nabla_{\mathbb{H}}^2v(q))^*=\frac{1}{2}\left( \nabla_{\mathbb{H}}^2v(q)+(\nabla_{\mathbb{H}}^2v(q))^t\right)
$.
Our starting point is the  Taylor expansion for  a function $v\in\mathcal{C}^2(\mathbb{H})$ at a point $q=(x,y,z)\in\mathbb{H}$ adapted to the Heisenberg group,  which we take from Section 3 in
\cite{LMR20},
\begin{equation}\label{taylor1}
v(p) = v(q) + \langle (\nabla_{\mathbb{H}},
Z) v(q), q^{-1}*p\rangle+
 \frac{1}{2} \big\langle (\nabla_{\mathbb{H}}^2v(q))^* \overline{q^{-1}*p},\overline{q^{-1}*p}  \big\rangle + o(|q^{-1}* p|_K^2).
\end{equation}
Next, we consider the horizontal average operator
$$
{\mathcal{A}}_2(v, \varepsilon)(q)  = \dashint_{B^2_ \varepsilon(0)}v(q*(a,b,0))\,da\,db
= \dashint_{B^2_1(0)}v\big(q+  \varepsilon(a,b,\frac{1}{2}(xb-ya))\big) \,da\,db.
$$
From Proposition 2.3 in \cite{LMR20} we get 
\begin{lemma}\label{2dexpansion}
 $$
 {\mathcal{A}}_2(v, \varepsilon)(q)  - v(q) = \frac{ \varepsilon^2}{8}\Delta_{\mathbb{H}} v(q)+o( \varepsilon^2),
\qquad\textrm{ as $ \varepsilon\to0$. }
 $$
 \end{lemma}

By changing variables we obtain  averages over general horizontal planes.

\begin{lemma}\label{2dexpansionA} Let $A$ be a $2\times2$ matrix and let $\eps>0$.   We have
 $$ \dashint_{B^2_1(0)}v(q*\eps (A\cdot(a,b),0))\,da\,db -v(q) =\frac{\eps^2}{8}\tr(A^t(\nabla_{\mathbb{H}}^2v(q))^*A) +o( \eps^2 |A|^2)$$
as $\varepsilon\to0$. 
 \end{lemma}

\begin{proof}Set 
$p=q*\eps (A\cdot(a, b), 0)$, where $(a,b)\in B^2_1(0)$. Observe that $q^{-1}*p=\eps(A\cdot(a,b), 0)$ so that $\overline{q^{-1}*p}= \eps A\cdot(a,b)$. Writing the Taylor expansion \eqref{taylor1} we get
\begin{align*}
v(p) & =   v(q) +  \langle( \nabla_{\mathbb{H}}, Z) v(q), (\eps A\cdot (a,b),0)\rangle \\
& \qquad + 
 \frac{\eps^2}{2} \langle (\nabla_{\mathbb{H}}^2v(q))^*\cdot A\cdot (a,b),A\cdot (a,b)  \rangle + o(|\eps A\cdot (a,b),0)|_K^2)\\
 &= v(q) + \eps \, \langle \nabla_{\mathbb{H}} v(q),  A\cdot (a,b)\rangle+
 \frac{\eps^2}{2} \langle (A^t \cdot \nabla_{\mathbb{H}}^2v(q))^*\cdot A\cdot (a,b),(a,b)  \rangle 
 +o (|\eps A\cdot (a,b)|^2).\qedhere
 \end{align*}
\end{proof}
We now state a version of Theorem \ref{thm.main3.intro} in the Heisenberg group  for the expression $$F(D^2v(q))= 2 \left( \det (\nabla_{\mathbb{H}}^2v(q))^*\right)^{1/2}.$$
Recall that we have
\[
\inf_{A\in\mathcal{A}} \textrm{trace}(A^tMA)=
\left\{
\begin{aligned}
&2\left(
\det{M}\right)^{1/2} && \textrm{if}\ M\geq0
\\
& -\infty& &\textrm{otherwise,}
\end{aligned}
\right.
\]
for the unbounded set $\mathcal{A}=\{A\in S_+^2(\mathbb{R}): \det(A)=~1\}$, for example 
 see \cite{[Blanc et al. 2020]}. We say that a function $v\in C^2$ is horizontally convex if
$(\nabla_{\mathbb{H}}^2v(q))^*\ge 0$. 

\begin{theorem} Let $v\in C^2$ be horizontally convex and $\phi(\eps)$ a positive function satisfying \eqref{hypothesis.phi.intro}. Then, we have 
$$
\mathop{\mathop{\inf}_{\det(A)=1}}_{A\leq \phi(\varepsilon)I}  \dashint_{B^2_1(0)} v(q*(A(a,b),0))\,da\, db
-v(q)= \frac{\eps^2}{4}\left( \det (\nabla_{\mathbb{H}}^2v(q))^*\right)^{1/2}+ o(\eps^2),
\qquad\textrm{as $\eps\to0$.}
$$
\end{theorem}
The proof of this theorem is a straightforward adaptation of the proof of Theorem 1.1 in \cite{[Blanc et al. 2020]}. Note that the two-dimensional averages are taken over ellipsoids in the horizontal tangent plane $T_q$ defined in \eqref{horizontalplane}.

%---------------------------------
% BIBLIOGRAFIA
%---------------------------------

\bibliographystyle{plain}

\end{document}